\documentclass{amsart}

\usepackage{amssymb, amsmath, amsthm, hyperref}

\usepackage{amsmath, amscd}

\usepackage{graphicx}

\usepackage[english]{babel}

\author{Rafael Torres and Jonathan Yazinski}

\title[Symplectic geography in dimensions 4 and 6]{Geography of symplectic 4- and 6-manifolds.}

\address{Mathematical Institute - University of Oxford\\ 24-29 St Giles'\\Oxford\\OX1 3LB\\England}

\email{torres@maths.ox.ac.uk}

\address{McMaster University - Dept. Mathematics and Statistics\\ 1280 Main Street West\\L8S 4K1\\Hamilton, ON.}

\email{yazinski@math.mcmaster.ca}

\date{December 15, 2010. Revised: July 10, 2011}



\thanks{\`A toi: tu me manques encore, repose en paix.}

\theoremstyle{plain}

\newtheorem{theorem}[equation]{Theorem}

\newtheorem{corollary}[equation]{Corollary}

\newtheorem{proposition}[equation]{Proposition}

\newtheorem{lemma}[equation]{Lemma}

\newtheorem{remark}[equation]{Remark}

\theoremstyle{definition}

\newtheorem{definition}[equation]{Definition}

\newtheorem{example}[equation]{Example}

\newcommand{\Q}{\mathbb{Q}}

\newcommand{\R}{\mathbb{R}}

\newcommand{\Z}{\mathbb{Z}}

\newcommand{\N}{\mathbb{N}}

\begin{document}

\maketitle

The geography of minimal symplectic 4-manifolds with arbitrary fundamental group and symplectic 6-manifolds with abelian fundamental group of small rank, and with arbitrary fundamental group are addressed.


\section{Introduction}{{\label{Intro}}


The geography problem for simply connected irreducible symplectic 4-manifolds with negative signature has been settled in its entirety \cite{[Go], [FS], [PS], [JP1], [A1], [FPS], [BK0], [BK], [ABBKP], [AP]}. Much less is known about the nonsimply connected realm. It has been observed that the new construction techniques \cite{[Go], [Lu], [ADK]} prove to be rather effective when one wishes to study the symplectic geography problem for a variety of choices of fundamental groups \cite{[Go], [BK], [ABBKP], [AP], [T0], [T1]}.\\

Concerning symplectic 4-manifolds with prescribed fundamental group, we obtained the following theorem.\\

\begin{theorem}{\label{Theorem 1}} Let $G$ be a group with a presentation consisting of $g$ generators and $r$ relations. Let $e, \sigma$ be integers that satisfy $2e + 3\sigma \geq 0, e + \sigma \equiv 0$ mod 4, $e + \sigma \geq 8$, and assume $\sigma \leq - 1$. Then there exists a minimal symplectic 4-manifold $M(G)$ with fundamental group $\pi_1(M(G)) \cong G$, characteristic numbers
\begin{center}
 $(c_1^2(M(G)), \chi_h(M(G))) = (2e + 3\sigma +4(g + r), \frac{1}{4}(e + \sigma) + (g + r))$,

\end{center}
and with odd intersection form over $\Z$. If $G$ is finite, then $M(G)$ admits exotic smooth structures.\\

Assume $s\geq 1, n\in \N$. For each of the following pairs of integers
\begin{center}
$(c, \chi) = (8n - 8 + 8(g + r), 2s + n - 1 + (g + r))$,
\end{center}
there exists a spin symplectic 4-manifold $X(G)$ with $\pi_1(X(G)) \cong G$, and characteristic numbers
\begin{center}
$(c_1^2(X(G)), \chi_h(X(G))) = (c, \chi)$.
\end{center}
Moreover, there is an infinite set $\{X_m(G)): m\in \N\} $ of pairwise nondiffeomorphic symplectic and nonsymplectic manifolds homeomorphic to $X(G)$.\\
\end{theorem}

If $G$ is a residually finite group, the work of M. J. D. Hamilton and D. Kotschick \cite{[HKo]} implies that the manifolds constructed are irreducible. Theorem \ref{Theorem 1} extends earlier results \cite{[Go], [BK1], [JP], [BK]}; we give a brief account of them in order to make clear the contribution of the previous theorem. R. E. Gompf proved the well-known result that any finitely presented group can be the fundamental group of a symplectic $2n$-manifold for $n\geq 2$ \cite[Section 6]{[Go]}. S. Baldridge and P. Kirk took on the endeavor of minimizing the Euler characteristic with respect to the generators and relations of a presentation of the fundamental group in \cite{[BK1]}, using the elliptic surface $E(1)$ and the manifold described in Section \ref{Section 2.2} as building blocks. They constructed a minimal symplectic manifold with Euler characteristic $12 + 12(g + r)$, and signature $- 8 -8(g + r)$.  J. Park systematically studied the geography in \cite[Theorem 1]{[JP]} using the construction in \cite{[BK1]} as a building block. Moreover, his result also addresses the existence of a myriad of exotic smooth structures by establishing a homeomorphism on the manifolds he constructed, using the results of S. Boyer \cite{[B]} on homeomorphisms of simply connected 4-manifolds with a given boundary.\\

Using a minimal symplectic manifold homeomorphic to $3\mathbb{CP}^2 \# 5 \overline{\mathbb{CP}^2}$ as a building block instead of $E(1)$,  Baldridge and Kirk constructed a smaller minimal symplectic example with Euler characteristic $10 + 6(g + r)$, and signature $-2 - 2(g + r)$ in \cite[Theorem 24]{[BK]} ($(c_1^2, \chi_h) = (14 + 6(g + r), 2 + (g + r))$ under the coordinates used in the statement of Theorem \ref{Theorem 1}), . In \cite{[Y]}, the second author of the present paper improved the known bounds by producing a minimal symplectic example with Euler characteristic $10 + 4(g + r)$, and signature $-2$ $((c_1^2, \chi_h) = (14 + 8(g + r), 2 + (g + r)))$. His construction uses as a building block the symplectic manifold sharing the homology of $S^2\times S^2$ built by  R. Fintushel, B. D. Park, and R. J. Stern in \cite[Section 4]{[FPS]}; see Section \ref{Section 2.2}. Theorem \ref{Theorem 1} offers an improvement and extends the results in \cite{[BK1], [JP], [BK], [Y]}.\\

For known fundamental-group-dependent bounds on the geographical regions that can be populated by minimal symplectic manifolds with prescribed fundamental group and comparisons to fundamental-group-dependent results on the existence of smooth manifolds in dimension four, the reader is directed to the work of C. H. Taubes \cite{[Ta1]}, of S. Baldridge and P. Kirk \cite{[BK1]}, of D. Kotschick \cite{[DK]}, and of P. Kirk and C. Livingston \cite{[KL], [KL1]}.\\

We now direct our attention to high dimensional manifolds, where existence results for symplectic 6-manifolds is somewhat uncharted territory. The geography of simply connected symplectic 6-manifolds was settled by M. Halic in \cite{[MH]} (see also \cite{[BH], [OV]}), where he proved that any triple $(a, b, c) \in \Z \oplus \Z \oplus \Z$ satisfying the arithmetic conditions $a \equiv c \equiv 0$ mod 2, and $b \equiv 0$ mod 24 can be realized by a simply connected symplectic 6-manifold with Chern numbers $(c_1^3, c_1c_2, c_3) = (a, b, c)$.  Although there is no established definition of minimality in dimension six yet (see \cite[Definition 5.2]{[TY]}), the existence result of Halic relies heavily on blow ups along points and along surfaces of well chosen examples. In particular,  the manifolds constructed in \cite{[MH]} are not minimal (\cite[Lemma 4.1]{[MH]}.) This scenario motivates interest for constructions of examples  that do not involve blow ups. In this direction, we obtained the following two results.\\

First, by building greatly upon the aforementioned work on the geography of simply connected 4-manifolds, we have the following theorem. 


\begin{theorem}{\label{Theorem 2}} Let $G$ be a group among the choices $\{\Z_p, \Z_p \oplus \Z_q, \Z_p \oplus \Z, \Z, \pi_1(\Sigma_g), F_n\}$, where $\Sigma_g$ is a closed oriented surface of genus $g$, and $F_n$ is the free group of rank $n$. Let $e_i, \sigma_i$ $i =1, 2$ be integers that satisfy $2e_i + 3 \sigma_i \geq 0, e_i + \sigma_i \equiv 0$ mod 4, and $e_i + \sigma_i \geq 8$. There exist symplectic 6-manifolds $W_0(G), W_1(G), W_2(G)$ with fundamental group $\pi_1(W_i(G)) \cong G$ for $i = 0, 1, 2$, and the following Chern numbers:

\begin{enumerate}

\item 
$c_1^3(W_0(G)) = 18\cdot (\sigma_1 + \sigma_2) + 12\cdot (e_1 + e_2)$\\ 
$c_1c_2(W_0(G)) = 6 \cdot (e_1 + \sigma_1 + e_2 + \sigma_2)$, and\\
$c_3(W_0(G)) = 2\cdot (e_1 + e_2)$;\\

\item 
$c_1^3(W_1(G)) = 0$,\\
$c_1c_2(W_1(G)) = 0$, and\\
$c_3(W_1(G)) = 0$;\\

\item 
$c_1^3(W_2(G)) =  - 18\cdot (\sigma_1 + \sigma_2) - 12 \cdot (e_1 + e_2) - 48$,\\ 
$c_1c_2(W_2(G)) =  - 6 \cdot (e_1 + \sigma_1 + e_2 + \sigma_2) - 24$, and\\
$c_3(W_2(G)) =  - 2\cdot (e_1 + e_2) - 8$.\\
\end{enumerate}

Assume $s_1, s_2\geq 1$, and $n_1, n_2\in \N$. There exist spin symplectic 6-manifolds $Y_i(G)$ for $i = 0, 1, 2$ with fundamental group $\pi_1(Y_i(G)) \cong G$, and whose Chern numbers are\\
\begin{enumerate}
\item 
$c_1^3(Y_0(G)) = 48 \cdot (n_1 + n_2 - 2)$,\\
$c_1c_2(Y_0(G)) = 48 \cdot (s_1 + s_2) + 24\cdot (n_1 + n_2 - 2)$, and\\
$c_3(Y_0(G)) = 48\cdot (s_1 + s_2) + 8 \cdot (n_1 + n_2 - 2)$;\\

\item 
$c_1^3(Y_1(G)) = 0$,\\
$c_1c_2(Y_1(G)) = 0$, and\\
$c_3(Y_1(G)) = 0$;\\

\item 
$c_1^3(Y_2(G)) =  - 48 \cdot (n_1 + n_2 - 1)$,\\
$c_1c_2(Y_2(G)) =  - 48 \cdot (s_1 + s_2) - 24 \cdot (n_1 + n_2 - 1)$, and\\
$c_3(Y_2(G)) =  - 48\cdot (s_1 + s_2) - 8(n_1 + n_2 - 1)$.\\
\end{enumerate}

\end{theorem}

Next, the current understanding on the geography of minimal symplectic 4-manifolds with arbitrary fundamental group (\cite{[Go], [BK1], [JP], [Y]}, Theorem \ref{Theorem 1}) is used to study the existence of symplectic 6-manifolds with prescribed fundamental group. Such enterprise had been pursued previously by R. E. Gompf in his seminal paper \cite[Theorem 7.1]{[Go]}. The result obtained in this direction is the following theorem.

\begin{theorem}{\label{Theorem 4}}   Let $G$ be a group with a presentation that consists of $g$ generators and $r$ relations. Let $e_i, \sigma_i$ $i =1, 2$ be integers that satisfy $2e_i + 3 \sigma_i \geq 0, e_i + \sigma_i \equiv 0$ mod 4, and $e_i + \sigma_i \geq 8$. There exist symplectic 6-manifolds $W_0(G), W_1(G), W_2(G)$ with fundamental group $\pi_1(W_i(G)) \cong G$ for $i = 0, 1, 2$, and the following Chern numbers:
\begin{enumerate}

\item 
$c_1^3(W_0(G)) = 18\cdot (\sigma_1 + \sigma_2) + 12 \cdot (e_1 + e_2) + 48 \cdot (g + r)$,\\ 
$c_1c_2(W_0(G)) = 6\cdot (e_1 + \sigma_1 + e_2 + \sigma_2) + 24 \cdot (g + r)$, and\\
$c_3(W_0(G)) = 2\cdot (e_1 + e_2) + 8 \cdot (g + r)$;\\

\item 
$c_1^3(W_1(G)) = 0$,\\
$c_1c_2(W_1(G)) = 0$, and\\
$c_3(W_1(G)) = 0$;\\

\item 
$c_1^3(W_2(G)) =  - 18\cdot (\sigma_1 + \sigma_2) - 12 \cdot (e_1 + e_2) - 48 \cdot (g + r) - 48$,\\ 
$c_1c_2(W_2(G)) =  - 6 \cdot (e_1 + \sigma_1 + e_2 + \sigma_2) - 24 \cdot (g + r) - 24$, and\\
$c_3(W_2(G)) =  - 2\cdot (e_1 + e_2) - 8 \cdot (g + r)$;\\
\end{enumerate}

Assume $s_1, s_2\geq 1$, and $n_1, n_2\in \N$. There exist spin symplectic 6-manifolds $Y_i(G)$ for $i = 0, 1, 2$ with fundamental group $\pi_1(Y_i(G)) \cong G$, and whose Chern numbers are\\
\begin{enumerate}
\item 
$c_1^3(Y_0(G)) = 48\cdot (n_1 + n_2 - 2) + 48(g + r)$,\\
$c_1c_2(Y_0(G)) = 48\cdot (s_1 + s_2) + 24\cdot (n_1 + n_2 - 2) + 24 \cdot (g + r)$, and\\
$c_3(Y_0(G)) = 48\cdot (s_1 + s_2) + 8\cdot (n_1 + n_2 - 2) + 8\cdot (g + r)$;\\

\item 
$c_1^3(Y_1(G)) = 0$,\\
$c_1c_2(Y_1(G)) = 0$, and\\
$c_3(Y_1(G)) = 0$;\\

\item 
$c_1^3(Y_2(G)) =  - 48\cdot (n_1 + n_2 - 2) - 48 \cdot (g + r) - 48$,\\
$c_1c_2(Y_2(G)) =  - 24\cdot (n_1 + n_2 - 2) - 48 \cdot (s_1 + s_2) - 24 \cdot (g + r) - 24$, and\\
$c_3(Y_2(G)) =  -48 \cdot (s_1 + s_2) - 8 \cdot (n_1 + n_2 - 2) - 8\cdot (g + r) - 8$.\\
\end{enumerate}

\end{theorem}

A particular class of manifolds whose interest has sparkled the introduction of new techniques of symplectic constructions, are symplectic manifolds with canonical class $c_1 = 0$ that are known as Calabi-Yau manifolds. The special attention that these manifolds have received in recent years is of relevance for geographical studies. In dimension four, there is a fairly broad knowledge on the existence of these objects \cite[Theorem 1.2]{[Li1]}, \cite[Section 4]{[Li2]}. The earlier known 6-dimensional examples were K\"ahler manifolds constructed from algebraic geometry, like the quintic hypersurface inside the complex projective 4-space $\mathbb{CP}^4$. Symplectic nilmanifolds appeared as the first examples of non-K\"ahler symplectic Calabi-Yau manifolds (see \cite{[T], [CFG]}); these manifolds are quotients of nilpotent Lie groups with left-invariant symplectic forms.\\

In the paper \cite{[STY]}, I. Smith, R. Thomas, and S. T. Yau produced several symplectic 6-manifolds with vanishing canonical class, which are potentially non-K\"ahler. The first example of a simply connected symplectic, yet non-K\"ahler Calabi-Yau 6-manifold was constructed by J. Fine and D. Panov in \cite{[PF2]} (see \cite{[PF1]} as well). Fine and Panov employ hyperbolic geometry to construct their examples. A new surgical procedure called coisotropic Luttinger surgery was introduced by S. Baldridge and P. Kirk to construct such manifolds  in \cite{[BK6]}, where they also constructed several non-simply connected examples. \\


Several manifolds of Theorem \ref{Theorem 2} and Theorem \ref{Theorem 4} (Item (2)) satisfy $c_1 = 0$; this case is studied by A. Akhmedov in \cite{[A6]}. In particular, the first examples of manifolds $Y_1(\{1\}), Y_1(\Z)$, 
in item (2) of Theorem \ref{Theorem 4} with $c_1 = 0$ appeared in \cite[Theorem 1, Theorem 2]{[A6]}. These manifolds have large third Betti number. In \cite{[DFN]}, Fine and Panov constructed Calabi Yau manifolds with arbitrary fundamental group by resolving orbifolds. In particular, they constructed examples of such manifolds with $b_3 = 0$ \cite[Theorem 3]{[DFN]}.\\

Finally, we recover and extend the main result in \cite{[MH]} to the geography of symplectic 6-manifolds for various choices of fundamental groups. The case of arbitrary fundamental group was studied previously by R. E. Gompf in \cite[Section 7]{[Go]}, and by F. Pasquotto in \cite[Proposition 4.31]{[BP]}. The statement is the following.\\

\begin{corollary}{\label{Corollary 3}}Let $X(G)$ be a symplectic 6-manifold. Then, 
\begin{center}
$c_1^3(X(G)) \equiv c_3(X(G)) \equiv 0$ mod 2,\end{center} \begin{center} $c_1c_2(X(G)) \equiv 0 $ mod 24, and $\pi_1(X(G)) \cong G$.
\end{center}
Conversely, any triple $(a, b, c)$ of integers satisfying $a \equiv c \equiv 0$ mod 2 and $b \equiv 0$ mod 24 occurs as a triple $(c_1^3, c_1c_2, c_3)$ of Chern numbers of a symplectic 6-manifold $X(G)$ with $\pi_1(X(G)) \cong G$, where $G$ is a finitely presented group.\\

\end{corollary}

Unlike the manifolds of Theorem \ref{Theorem 2} and Theorem \ref{Theorem 4}, the construction of the manifolds of  Corollary \ref{Corollary 3} involve blow ups along points and along surfaces. The fundamental tools of construction used to obtain all these results are R. E. Gompf's symplectic sum \cite{[Go]} and Luttinger surgery \cite{[Lu], [ADK]}.\\ 

The paper is organized as follows. The second section contains a description of our 4-dimensional building blocks. In Section \ref{Section 2.1}, we present the building blocks that allows us to produce manifolds for several choices of fundamental groups. The building block used to construct manifolds with arbitrary fundamental group is described in Section \ref{Section 2.2}. A useful explanations on the fundamental group computations of Theorem \ref{Theorem 1} is given in Section \ref{Section 2.3}. The fundamental group computations are located in Section \ref{Section 5}. Section \ref{Section 2.4} and Section \ref{Section 2.5}  contain the building blocks we use to populate geographic regions. The geography of symplectic 4-manifolds with prescribed fundamental group is studied in Section \ref{Section 3}. The proof of Theorem \ref{Theorem 1} is given in Section \ref{Section 3.1} and Section \ref{Section 3.2}. We then use these results to study 6-manifolds in Section \ref{Section 4}. The formulae needed to compute Chern numbers are given in Section \ref{Section 4.1}. In Section \ref{Section 4.2}, we build symplectic sums that are fundamental for our purposes. The tool that allow us to proof Corollary \ref{Corollary 3} is presented in Section \ref{Section 4.3}, where we work out an example to demonstrate the utility of these techniques for several fundamental groups. Section \ref{Section 5.1} provides a detailed description on our choices of gluing maps used in the 6-dimensional symplectic sums. The fundamental group computations involved in the statements of several of our results, including Theorem \ref{Theorem 2}, and Theorem \ref{Theorem 4}, are presented in Section \ref{Section 5.2}. Finally, Section \ref{Section 5.3} contains the proofs of Theorem \ref{Theorem 4} and Corollary \ref{Corollary 3}.

\section{Building blocks}{\label{Section 2}}

\subsection{Symplectic 4-manifolds used to obtain several of our choices for $\pi_1$.}{\label{Section 2.1}} The following building block in our symplectic sums will allow us to vary our choices of fundamental groups without contributing anything to the characteristic numbers. Take the product of a torus and a genus two surface $T^2\times \Sigma_2$. The characteristic numbers are given by $c_2(T^2\times \Sigma_2) = 0 = \sigma(T^2\times \Sigma_2)$. Endow $T^2\times \Sigma_2$ with the symplectic product form. The torus $T^2\times \{x\}$ and the genus two surface $\{x\}\times \Sigma_2$ are symplectic submanifolds of $T^2\times \Sigma_2$, which are geometrically dual to each other. Let $F$ denote a parallel copy of the surface of genus two. Furthermore, $T^2\times \Sigma_2$ contains 4 pairs of geometrically dual Lagrangian tori available to perform Luttinger surgery \cite{[Lu], [ADK]} on them; see also \cite[Section 2]{[HL]}. These tori are of the form $S^1\times S^1 \subset T^2\times \Sigma_2$. \\

Regarding the very relevant fundamental group computations, we have the following result. Let $\{x, y\}$ be the generators of $\pi_1(T^2)$, and let $\{a_1, b_1, a_2, b_2\}$ be the standard set of generators of $\pi_1(\Sigma_2)$.\\

\begin{proposition}{\label{Proposition 5}} (Baldridge - Kirk  \cite[Proposition 7]{[BK]}). The fundamental group
\begin{center}$\pi_1(T^2\times \Sigma_2 - (T_1 \cup T_2 \cup T_3 \cup T_4 \cup F))$\end{center}
is generated by the loops  $x, y, a_1, b_1, a_2, b_2$. Moreover,  with respect to certain paths to the boundary of the tubular neighborhoods of the $T_i$ and $F$, the meridians and two Lagrangian push offs of the surfaces are given by
\begin{itemize}
\item $T_1: m_1 = x, l_1 = a_1$, $\mu_1= [b^{-1}, y^{-1}]$,
\item $T_2:  m_2 = y, l_2 = b_1a_1b^{-1},$ $\mu_2 = [x^{-1}, b_1]$,
\item $T_3:  m_3 = x, l_3 = a_2$, $\mu_3 = [b_2^{-1}, y^{-1}]$,
\item $T_4: m_4 = y, l_4 = b_2a_2b_2^{-1}$, $\mu_4 = [x^{-1}, b_2]$,
\item $\mu_{F} = [x, y]$.
\end{itemize}
The loops $a_1, b_1, a_2, b_2$ lie on the genus 2 surface and form a standard set of generators; the relation $[a_1, b_1][a_2, b_2] = 1$ holds.
\end{proposition}

The reader is referred to the lovely paper \cite{[BK]} for details on constructions of symplectic 4-manifolds using Luttinger surgery. We finish this section by pointing out that the description above yields the building blocks needed to produce the manifolds of Theorem \ref{Theorem 1}, Theorem \ref{Theorem 2}, and Corollary \ref{Corollary 3} with surface fundamental groups $\pi_1(\Sigma_g)$. Take the product symplectic manifold $T^2\times \Sigma_g$ of the 2-torus with a surface of genus $g\geq 3$. Endowed with the product symplectic form, the submanifold $T^2\times \{x\} \subset T^2\times \Sigma_g$ is symplectic. Given that $\{x\} \times \Sigma_g$ is geometrically dual to it, it is straight-forward to see that the meridian $\mu_{T^2}$ in $T^2\times \Sigma_g - T^2\times \{x\}$ can be expressed in terms of the generators and relations of $\pi_1(\{x\}\times \Sigma_g)$ see \cite{[BK0], [BK]} for details.

\subsection{4-dimensional building block with prescribed fundamental group.}{\label{Section 2.2}}

In \cite[Theorem 1.1]{[Y]}, the second author proved the following result.

\begin{theorem}{\label{Theorem Y}}  Let $G$ be a group with a presentation with $g$ generators and $r$ relations. There exists a minimal symplectic 4-manifold $X$ with fundamental group $\pi_1(X(G)) \cong G$, and characteristic numbers  \begin{center}$c_1^2(X(G)) = 14 + 8(g + r), c_2(X(G)) =  10 + 4(g + r)$, and $\sigma(X(G)) = -2$, $\chi_h(X(G)) = 2 + (g + r)$.\end{center}

The manifold $X(G)$ contains  a homologically essential Lagrangian torus $T$, and a symplectic surface of genus two $F$ with trivial normal bundle, and \begin{center}$\pi_1(X(G) - T) \cong \pi_1(X(G) - F) \cong \pi_1(X(G)) \cong G$ .\\\end{center}
\end{theorem}

Let us describe the materials used to prove this result. The fundamental building block for these constructions is a symplectic manifold built in \cite{[BK1]}, which allows the manipulation on the number of generators and relations on fundamental groups. We proceed to describe it.\\

S. Baldridge and P. Kirk take a 3-manifold $Y$ that fibers over the circle, and build the 4-manifold $N: = Y\times S^1$. It admits a symplectic structure (p. 856 in \cite{[BK1]}). Its Euler characteristic and its signature are both zero.\\

The fundamental group $\pi_1(Y\times S^1)$ has classes $s, t, \gamma_1, \ldots, \gamma_{g + r}$ so that
\begin{center}
$G \cong \pi_1(Y\times S^1) / N(s, t, \gamma_1, \ldots, \gamma_{g + r})$,
\end{center}

where $N(s, t, \gamma_1, \ldots, \gamma_{g + r})$ is the normal subgroup generated by the aforementioned classes.\\

A result of W. Thurston \cite{[T]} implies that the manifold $N$ is symplectic. There are $g + r + 1$ symplectically imbedded tori $T_0, T_1, \ldots, T_{g + r} \subset N$, which have the following crucial traits for our purposes concerning fundamental group computations.
\begin{itemize}

\item The generators of $\pi_1(T_0)$ represent $s$ and $t$, and
\item the generators of $\pi_1(T_i)$ represent $s$ and $\gamma_i$ for $i \geq 1$.
\end{itemize}

The curve $s$ has the form $\{y\} \times S^1 \subset Y\times S^1$, with $y\in Y$, and $\gamma_i$ has the form $\gamma_i \times \{x\}\subset Y\times \{x\}, x\in S^1$.  We point out that the role of the curves $\gamma_i$ is to provide the $r$ relations in a presentation of $G$ \cite[Section 4]{[BK1]}.\\

The details of the construction depend on the presentation
of our group; we give a description of the construction with the purpose of making the paper a bit more self-contained. Let $G$ be a finitely presented group:
$G=\left\langle x_{1},\ldots,x_{g}|w_{1},\ldots,w_{r}\right\rangle$
with $g$ generators $x_{1},\ldots,x_{g}$ and $r$ relations $w_{1},\ldots
,w_{r}.$ Let $n_{i}$ denote the length of the relation $w_{i},$ and set
\[n=1+{\displaystyle\sum_{i=1}^{r}}n_{i}\]

Let $\Sigma_{gn}$ denote a surface of genus $gn.$ We will define an
automorphism $R$ on $\Sigma_{gn},$ and so we consider the following
description of $\Sigma_{gn}$ to aid with specifying $R.$ Consider a round $2
$-sphere $S^{2}\subseteq\mathbb{R}^{3}$ centered at the origin, and let
$\bar{R}$ denote a rotation of $S^{2}$ by an angle of $\frac{2\pi}{ng}$
through the axis passing through the points $\left(  0,0,\pm1\right)  .$ The
orbit of $\left(  1,0,0\right)  $ under $\left\langle \bar{R}\right\rangle
\subseteq SO\left(  3\right)  $ gives $ng$ distinct points on $S^{2},$ and
performing a connected sum of a torus with $S^{2}$ at each point in
$\left\langle \bar{R}\right\rangle \left\{  \left(  1,0,0\right)  \right\}  $
in an equivariant way gives a surface $\Sigma_{gn}$ so that the action of
$\bar{R}$ on $S^{2}$ naturally gives an action $R$ of $\mathbb{Z}/\left(
gn\right) $ on $\Sigma_{gn}.$

Let $Y$ be a surface bundle over $S^{1}$ given by the monodromy $R^{g}.$ That
is, let
\[
Y=\frac{F\times I}{\left(  x,0\right)  \sim\left(  R^{g}\left(  x\right)
,1\right)  }.
\]
Let $p=\left(  0,0,1\right)  ,$ and let $\tau$ denote the loop in $Y$ that is
the image in $Y$ of $\left\{  p\right\}  \times I$ under $\sim.$ This is a
loop since $R^{g}\left(  p\right)  =p.$ Let $s,t\in\pi_{1}\left(  Y\times
S^{1}\right)  $ be given as follows: $t$ is represented by a parametrization
of $\tau\times\left\{  1\right\}  ,$ and $s$ is a generator of $\left\{
1\right\}  \times\mathbb{Z}\subseteq\pi_{1}\left(  Y\right)  \times\pi
_{1}\left(  S^{1}\right)  .$

We now use the presentation of $G$ given previously to get
a new presentation of $G.$ We add $g$ generators $y_{1},\ldots,y_{g},$ we
replace the relations $w_{1},\ldots,w_{r}$ with new relations $\tilde{w}%
_{1},\ldots,\tilde{w}_{r},$ and we add $g$ more relations $x_{1}y_{1}%
,\ldots,x_{g}y_{g}.$ Hence we have a new presentation of the same group $G:$
\begin{equation}
G=\left\langle x_{1},\ldots,x_{g},y_{1},\ldots,y_{g}|\tilde{w}_{1}%
,\ldots,\tilde{w}_{r},x_{1}y_{1},\ldots,x_{g}y_{g}\right\rangle
\label{Num-NewPresentationofG}%
\end{equation}
The relation $\tilde{w}_{i}$ is obtained from the relation $w_{i}$ by
replacing each occurence of $x_{k}^{-1}$ with $y_{k}.$ Hence each relation is
expressed as a product of positive powers of the generators.

Each relation in the presentation of $G$ given in
(\ref{Num-NewPresentationofG}) gives an immersed curve $\gamma_{i}$ on the
surface $\Sigma_{gn}.$ Identify $\Sigma_{gn}$ with a fiber of $Y,$ and then
take the product of these curve with a circle to obtain tori $T_{i}%
:=\gamma_{i}\times S^{1}\subseteq Y\times S^{1}.$ Figure 1 below is taken from
\cite{[BK1]}, and here we see examples of how occurences of $x_{i}$ and $y_{i}$
give us curves on this surface.

\begin{figure}{\label{Figure 1}}
\begin{center}
\includegraphics[viewport= 170 5 150 310]{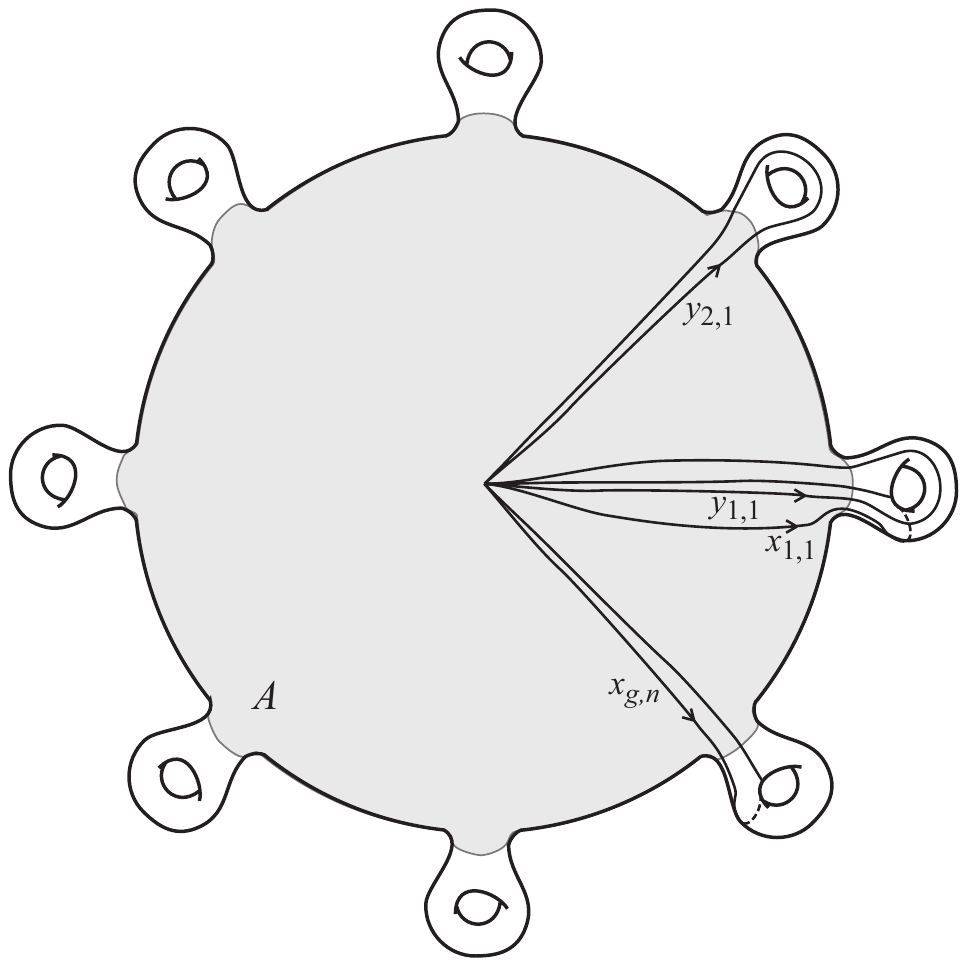}
\caption{}
\end{center}
\end{figure}


In order to make the construction work, the tori $T_{i}$ need to be disjoint
but immersed, and the curves $\gamma_{i}$ need to be mapped to different parts
of the surface. The details are provided in \cite{[BK1]}; here we only
briefly mention that by expressing the $r + g$ relations as only positive powers of
the $2g$ generators, one ensures that we can perturb the tori $T_{i}$ so that they
are symplectic.

A final torus that we consider is $T_{0},$ which is the product of the circle
corresponding to $s,t\in\pi_{1}\left(  Y\times S^{1}\right)  .$

The presentation for $\pi_{1}\left(  Y\times S^{1}\right)  $ is
\begin{equation}
\left\langle \pi_{1}\Sigma_{gn},t|R_{\ast}^{g}\left(  x\right)  =txt^{-1}%
\text{ for }x\in\pi_{1}\Sigma_{gn}\right\rangle \times\left\langle
s\right\rangle \label{Num-PresentationNtimesS1}%
\end{equation}
where $\pi_{1}\Sigma_{gn}$ has the presentation
\begin{equation}
\pi_{1}\Sigma_{gn}=\left\langle
\begin{array}
[c]{c}%
x_{1,1},y_{1,1},\ldots,x_{g,1},y_{g,1},x_{1,2},y_{1,2},\ldots,x_{g,2}%
,y_{g,2},\ldots,\\
x_{1,n},y_{1,n},\ldots,x_{g,n},y_{g,n}%
\end{array}
\left\vert
{\displaystyle\prod\limits_{l=1}^{n}}
{\displaystyle\prod\limits_{k=1}^{g}}
\left[  x_{k,l},y_{k,l}\right]  \right.  \right\rangle
\label{Num-PresenationPi1F}%
\end{equation}

We symplectic sum along the tori $T_{0},\ldots,T_{r+g}$ with other symplectic
manifolds to kill the pushoffs of the curves that generate $\pi_{1}\left(
T_{i}\right)  $ while not introducing additional generators to the fundamental
group of our manifold. For example, if we symplectic sum in such a manner
along $T_{0},$ then will kill the generators $s$ and $t,$ and the presentation
of (\ref{Num-PresentationNtimesS1}) reduces to
\[
\left\langle \pi_{1}\Sigma_{gn}|R_{\ast}^{g}\left(  x\right)  =x\text{ for
}x\in\pi_{1}\Sigma_{gn}\right\rangle
\]
Notice that in the notation of (\ref{Num-PresenationPi1F}), $R_{\ast}%
^{g}\left(  x_{k,l}\right)  =x_{k,l+1}$ and $R_{\ast}^{g}\left(
y_{k.l}\right)  =y_{k.l+1},$ where the addition on the second subscript is
taken modulo $n$ (with $n+1=1$), and so setting $x_{k}=x_{k,l}$ and
$y_{k}=y_{k,l}$ the presentation of (\ref{Num-PresentationNtimesS1}) actually
reduces further to,
\begin{equation}
\left\langle x_{1},y_{1},\ldots,x_{g},y_{g}\left\vert
{\displaystyle\prod\limits_{l=1}^{n}}
{\displaystyle\prod\limits_{k=1}^{g}}
\left[  x_{k},y_{k}\right]  \right.  \right\rangle
\label{Num-PresentationNtimesS1_2}%
\end{equation}
Summing along the tori $T_{r+1},\ldots,T_{r+g}$ introduced the relations
$x_{k}y_{k}=1,$ and so the commutators $\left[  x_{k},y_{k}\right]  $ are
trivial, so that the surface relation
\[
{\displaystyle\prod\limits_{l=1}^{n}}
{\displaystyle\prod\limits_{k=1}^{g}}
\left[  x_{k},y_{k}\right]
\]
of (\ref{Num-PresentationNtimesS1_2}) is trivial. Finally, performing a
symplectic sum along the tori $T_{1},\ldots,T_{r}$ introduces the relations
$\tilde{w}_{1},\ldots,\tilde{w}_{r},$ and the presentation of
(\ref{Num-PresentationNtimesS1_2}) is now
\[
\left\langle x_{1},y_{1},\ldots,x_{g},y_{g}|x_{1}y_{1},\ldots,x_{g}%
y_{g},\tilde{w}_{1},\ldots,\tilde{w}_{r}\right\rangle
\]
which is just the presentation of $G$ in (\ref{Num-NewPresentationofG}). Hence
we have shown that fiber summing $Y\times S^{1}$ with the appropriate
symplectic manifolds along the tori $T_{0},\ldots,T_{r+g}$ is a symplectic
manifold with fundamental group $G.$\\

For a proof of Theorem \ref{Theorem Y}], the second author used a spin symplectic 4-manifold $V$ with $c_2(V) = 4, \sigma = 0$ that shares the cohomology ring of $2(S^2\times S^2) \# S^3\times S^1$ and contains a symplectic torus $T_V$ of self intersection zero carrying a generator of $H_1(V; \Z) \cong \Z$. This manifold was constructed in \cite[Section 4]{[FPS]}; its construction consists of  applying seven Luttinger surgeries \cite{[Lu], [ADK]} to the product of two copies of a genus two surface $\Sigma_2 \times \Sigma_2$. There exists a symplectic surface of genus two and self intersection zero $F\hookrightarrow V$. \\

A careful analysis of fundamental groups computations under cut and paste constructions was done by the second author in \cite[Section 3, Section 4]{[Y]}. The main ingredient that yields an improvement of \cite[Theorem]{[BK]} is the symplectic sum \cite[Theorem 1.3]{[Go]} of $N$ and $g + r$ copies of $V$ along they symplectic tori $T_i$ and $T_{V_i}$ for $ 1\leq i \leq g + r$; see \cite{[Y]} for details. \\

Denote this building block by $BK$. Its characteristic numbers are computed using the well known formulae in \cite[p. 535]{[Go]} to be $c_2(BK) = 4(g + r), \sigma(BK) = 0$.\\

\subsection{Computations of $\pi_1$ Theorem \ref{Theorem 1}.}{\label{Section 2.3}} In order for the manifold $BK$ to be used successfully in the production of a manifold whose fundamental group has a presentation consisting of $g$ generators and $r$ relations, one builds a symplectic sum with a manifold $X$ that results in the resulting manifold having a fundamental group where the classes $s = 1 = t$ (where $s$ and $t$ are generators of $\pi_1(BK)$ as in Section \ref{Section 2.2}.)   This was initially observed by R.E. Gompf \cite{[Go]}, S. Baldridge and P. Kirk in \cite{[BK1]}, and used by the second author in the proof of Theorem \ref{Theorem Y}.\\

Suppose there is a minimal symplectic 4-manifold $X$, which contains a symplectic torus $T \hookrightarrow X$ of self intersection zero such that $\pi_1(X - T) \cong \{1\}$. \cite[Section 4]{[Y]}. Build the symplectic sum along tori
\begin{center}

$Z:= BK_{T_0 = T} X$.\\
\end{center}

In \cite[Section 3, Section 4]{[Y]} it is proven that there exists an isomorphism \begin{center} $\pi_1(Z) \cong \left<g_1, \ldots, g_g | r_1, \ldots, r_r\right> \cong G$.\end{center}

The 4-manifold $X$ used in the proof of Theorem \ref{Theorem Y} is a minimal symplectic manifold homeomorphic to $3\mathbb{CP}^2 \# 5 \overline{\mathbb{CP}^2}$ constructed in \cite[Theorem 18]{[BK]}. To prove Theorem \ref{Theorem 1}, we will vary the choice of $X$ used in the construction described previously, where are choices come from from Proposition \ref{Proposition 7} and Proposition \ref{Proposition 11} in Section \ref{Section 2.4}. The fundamental group computations  follow verbatim from the work of the second author in \cite{[Y]}.

\subsection{4-dimensional building blocks used to populate geographic regions I}{\label{Section 2.4}} We proceed to describe the building blocks that have been successfully used by  topologists to understand the geography and botany problems of symplectic 4-manifolds (see for example \cite{[Go], [ABBKP], [BK], [AP], [T0]}, and the references there).\\

\begin{proposition}{\label{Proposition 7}} Let $b^- \in \{4, 5, \ldots, 17, 18, 19\}$. There exists a minimal symplectic simply connected 4-manifold $X_{3, b^-}$ with second Betti number given by $b_2^+(X_{3, b^-}) = 3$ and $b_2^-(X_{3, b^-}) = b^-$, which contains two homologically essential Lagrangian tori $T_1$ and $T_2$ such that \begin{center} $\pi_1(X_{3, b^-} - (T_1 \cup T_2)) = \{1\}$.\end{center}

The characteristic numbers are $c_2(X_{3, b^-}) = 5 + b_2^-, \sigma = 3 - b^-$, and $c_1^2(X_{3, b^-}) = 10 + 2b^- + 9 - 3 b_2^- = 19 - b^-$.\\
\end{proposition}

Table 1 provides a blueprint on the construction of the manifolds in the statement of Proposition \ref{Proposition 7}, and the sources of the building blocks. The notation used in Table 1 is similar to the one of Proposition \ref{Proposition 7}; the manifolds $X_{1, n}$ are homeomorphic to $\mathbb{CP}^2\# n \overline{\mathbb{CP}^2}$. The elliptic surface used is denoted by $E(1) = \mathbb{CP}^2\# 9 \overline{\mathbb{CP}^2}$. The symplectic sum of $X$ and $Y$ along $Z$ is denoted by $X \#_{Z} Y$.  \\

\begin{table}[ht] {\label{Table 1}}
\caption{Minimal symplectic 4-manifolds with $b_2^+ = 3$.}
\centering
\begin{tabular}{c | c | c | c | c }
\hline\hline
$b_2^-$ & $(c_1^2, \chi_h, \sigma)$ & Symplectic sum/manifold. & $\#$ of Luttinger surgeries. & Reference. \\
\hline\hline
$4$   & $(15, 2, -1)$  & $X_{1, 2} \#_{\Sigma_2} T^2\times \Sigma_2$ & $two$ & \cite[Section 9]{[AP]}  \\
$5$    & $(14, 2, -2)$  & $X_{3, 5}$ & $none$ & \cite[Theorem 18]{[BK]} \\
$6$    & $(13, 2, -3)$ & $X_{1, 2} \#_{\Sigma_2} (T^4\# 2\overline{\mathbb{CP}^2})$ & $none$ & \cite[Section 9]{[AP]}\\
$7$    & $(12, 2, -4)$ &$X_{3, 7}$ & $none$ & \cite[Corollary 15]{[BK]} \\
$8$   & $(11, 2, -5)$ & $X_{1, 5}\#_{\Sigma_2} (T^4\# \overline{\mathbb{CP}^2})$ & $none$ & \cite[Sections 3 and 4]{[AP]}.\\
$9$  & $(10, 2, -6)$ & $X_{1, 5} \#_{\Sigma_2} (T^4 \# 2\overline{\mathbb{CP}^2})$ & $none$ & \cite[Corollary 16]{[BK]}.\\

$10$  &$(9, 2, -7)$ & $X_{1, 6} \#_{\Sigma_2} (T^4 \# 2 \overline{\mathbb{CP}^2})$ & $none$ & \cite[Lemma 15]{[AP]}  \\
$11$  & $(8, 2, -8)$ & $E(1) \#_{T^2} S$ & $six$ & \cite[Lemma 16]{[ABBKP]}  \\
$12$   & $(7, 2, -9)$ & $E(1) \#_{T^2} A$ & $none$ & Theorem \ref{Theorem 10} \\
$13$  & $(6, 2, -10)$ & $E(1) \#_{T^2} B$ & $none$ & Theorem \ref{Theorem 10}  \\
$14$  & $(5, 2, -11)$ & $\mathbb{CP}^2 \# 12 \overline{\mathbb{CP}^2} \#_{\Sigma_2} T^2\times \Sigma_2$ & $two$ & \cite[Building Block 5.6]{[Go]}  \\
$15$  & $(4, 2, -12)$ & $\mathbb{CP}^2 \# 13 \overline{\mathbb{CP}^2} \#_{\Sigma_2} T^2\times \Sigma_2$ & $two$ & \cite[Building Block 5.6]{[Go]}  \\
$16$ & $(3, 2, -13)$ & $\mathbb{CP}^2 \# 12 \overline{\mathbb{CP}^2} \#_{\Sigma_2} (T^4\# 2 \overline{\mathbb{CP}^2})$ & $none$ & \cite[Building Blocks 5.6 and 5.6] {[Go]}  \\
$17$ & $(2, 2, -14)$ & $\mathbb{CP}^2 \# 13 \overline{\mathbb{CP}^2} \#_{\Sigma_2} (T^4\# 2 \overline{\mathbb{CP}^2})$ & $none$ & \cite[Building Blocks 5.6 and 5.7]{[Go]}  \\
$18$ & $(1, 2, -15)$ & $S_{1, 1}$ & $none$ & \cite[Example 5.4]{[Go]} \\
$19$ & $(0, 2,  -16)$ & $E(1)\#_{T^2} T^4 \#_{T^2} E(1)$ & $none$ & \cite{[SZs]}  \\


\hline
\end{tabular}
\label{table:masspr}
\end{table}

The following building block was used in \cite{[ABBKP]} to fill in vast regions of the geography. 

\begin{definition}{\label{Definition 8}}  \cite[Definition 2]{[ABBKP]}. An ordered triple $(X, T_1, T_2)$ consisting of a symplectic 4-manifold $X$ and two disjointly 
embedded Lagrangian tori $T_1$ and $T_2$ is called a \emph{telescoping triple} if:
\begin{enumerate}
\item The tori $T_1$ and $T_2$ span a 2-dimensional subspace of $H_2(X; \R)$.
\item The fundamental group is given by $\pi_1(X) \cong \Z^2$ and the inclusion induces an isomorphism $\pi_1(X - (T_1 \cup T_2)) \rightarrow \pi_1(X)$. In particular, the meridians 
of the tori are trivial in $\pi_1(X - (T_1 \cup T_2))$.
\item The image of the homomorphism induced by the corresponding inclusion $\pi_1(T_1) \rightarrow \pi_1(X)$ is a summand $\Z \subset \pi_1(X)$.
\item The homomorphism induced by inclusion $\pi_1(T_2) \rightarrow \pi_1(X)$ is an isomorphism.
\end{enumerate}

\end{definition}

A smooth 4-manifold is minimal if it does not contain any sphere of self intersection $-1$. If $X$ is a minimal manifold, then telescoping triple is called \emph{minimal}. We emphasize the importance of the order  of the tori in the definition. The meridians $\mu_{T_1}$, $\mu_{T_2}$ in $\pi_1(X - (T_1 \cup T_2))$ are trivial and the relevant fundamental groups are abelian. The push-off of an oriented loop $\gamma \subset T_i$ into $X - (T_1 \cup T_2)$ with respect to any (Lagrangian) framing of the normal bundle of $T_i$ represents a well-defined element of 
$\pi_1(X - (T_1 \cup T_2) )$ which is independent of the choices of framing and base-point.\\

The first condition assures us that the Lagrangian tori $T_1$ and $T_2$ are linearly independent in $H_2(X; \R)$. This allows  for the symplectic form on $X$  to be slightly perturbed so that one of the $T_i$ remains Lagrangian while the other becomes symplectic \cite[Lemma 1.6]{[Go]}. Four our purposes, we will require for $T_1$ to be symplectic. \\



\begin{proposition}{\label{Proposition 9}} (cf. \cite[Proposition 3]{[ABBKP]}). Let $(X, T_1, T_2)$ and $(X', T_1', T_2')$ be two telescoping triples. Then for an appropriate 
gluing map the triple

\begin{center}
$(X \#_{T_2, T_1'} X', T_1, T_2')$
\end{center}
is again a telescoping triple. If $X$ and $X'$ are minimal symplectic 4-manifolds, then the resulting telescoping triple is minimal. The Euler characteristic and the signature of $X \#_{T_2, T_1'} X'$ are given by $c_2(X) + c_2(X')$ and $\sigma(X) + \sigma(X')$.\\
\end{proposition}

The proof of Proposition \ref{Proposition 9} follows from \cite{[Go], [ABBKP], [AP]}, and the claim about minimality follows from Usher's theorem \cite{[U]}. The results minding the existence of telescoping triples are gathered in the following theorem, which was proven in  \cite[Section 5]{[ABBKP]}, \cite[Lemma 6, Lemma 7]{[T1]}.

\begin{theorem}{\label{Theorem 10}}  Existence of telescoping triples.
\begin{itemize} 
\item There exists a minimal telescoping triple $(A, T_1, T_2)$ satisfying $c_2(A) = 5$, $\sigma(A) = - 1$.
\item For each $g\geq 0$, there exists a minimal telescoping triple $(B_g, T_1, T_2)$ satisfying $c_2(B_g) = 6 +
4g$, $\sigma(B_g) = - 2$.
\item There exists a minimal telescoping triple $(C, T_1, T_2)$ satisfying  $c_2(C) = 7$, $\sigma(C) = - 3$.
\item There exists a minimal telescoping triple $(D, T_1, T_2)$ satisfying $c_2(D) = 8$, $\sigma(D)
= - 4$.
\item There exists a minimal telescoping triple $(F, T_1, T_2)$ satisfying $c_2(F) = 10$, $\sigma(F) = - 6$.\\ 
\end{itemize}
\end{theorem}

\subsection{4-dimensional building blocks used to populate geographic regions II}{\label{Section 2.5}} The building blocks of the previous sections are used to produce our guiding building blocks we will use in the proofs of our main results by using R.E. Gompf's symplectic sum \cite[Theorem 1.3]{[Go]}. A precise statement is the following.\\

\begin{proposition}{\label{Proposition 11}} (cf. \cite[Theorem A]{[ABBKP]}, \cite[Theorem 1]{[AP]}). Let $e, \sigma$ be integers that satisfy $2e + 3\sigma \geq 0, e + \sigma \equiv 0$ mod 4, $e + \sigma \geq 8$, and assume $\sigma \leq -1$. Then there exist simply connected minimal symplectic 4-manifolds $Z_{1, 1}$ and $Z_{1, 2}$, both with 
Euler characteristic $e$, signature $\sigma$, and with odd intersection form over $\Z$. Moreover, the 4-manifold $Z_{1,1}$ contains two homologically essential Lagrangian tori $T_1$ and $T_2$ such that\\
\begin{center}
$\pi_1(Z_{1,1} - (T_1 \cup T_2)) \cong \{1\}$.\\
\end{center}
The 4-manifold $Z_{1,2}$ contains a homologically essential Lagrangian torus $T_1$, and a symplectic surface of genus two $F$ such that
\begin{center}
$\pi_1(Z_{1,2} - (T_1 \cup F)) \cong \{1\}$.\\
\end{center}
\end{proposition}

\begin{remark}{\label{Remark 12}} At this point we emphasize that Proposition \ref{Proposition 7} and Proposition \ref{Proposition 11} are fruits of the efforts of several topologists (see \cite{[BK0], [BK], [ABBKP], [AP], [T0]} and the references provided there). The observation regarding the existence of the submanifolds is crucial for our purposes, and it has not appeared in the literature previously. \\
\end{remark}

\begin{proof} The construction of the 4-manifold $Z_{1,1}$ is obtained by an iterated use of the telescoping triples of Theorem \ref{Theorem 10}, Proposition \ref{Proposition 7}, Proposition \ref{Proposition 5} (which cover most of the geographic points), and the results of \cite{[BK], [ABBKP], [AP]} and \cite[Section 4]{[FPS]}. The details are left to the reader (see \cite{[T0]} for details).

 The construction of the 4-manifold $Z_{1, 2}$ goes as follows. Build the symplectic sum $Z:= Z_{1, 1} \#_{T_1 = T^2} T^2 \times \Sigma_2$. Apply four Luttinger surgeries to obtain a simply connected symplectic 4-manifold $Z_{1, 2}$; see Proposition \ref{Proposition 5} for the fundamental group computations. The torus $T_2$ is contained in $Z_{1, 2}$, and so is a parallel copy of the genus two surface $\{x\}\times \Sigma_2 \subset T^2\times \Sigma_2$ that we now call $F$ \cite[Corollary 1.7]{[Go]}. The meridian $\mu_F$ is trivial in the complement by Proposition \ref{Proposition 5}.

\end{proof}

\section{Geography of symplectic 4-manifolds with prescribed fundamental group.}{\label{Section 3}}

In this Section we give a proof of Theorem \ref{Theorem 1}. The proof is split in two parts 

\subsection{Manifolds with odd intersection form and negative signature.}{\label{Section 3.1}} The first part of Theorem \ref{Theorem 1} is proven in this section. We build greatly upon the progress on the geography problem for nonspin simply connected symplectic 4-manifolds with negative signature \cite{[Go], [JP], [BK0], [BK], [FPS], [ABBKP],[AP]}, and the study of minimal symplectic 4-manifolds with arbitrary fundamental group \cite{[BK1], [Y]}.

\begin{theorem}{\label{Theorem 14}} Let $G$ be a group with a presentation consisting of $g$ generators and $r$ relations. Let $e, \sigma$ be integers that satisfy $2e + 3\sigma \geq 0, e + \sigma \equiv 0$ mod 4, $e + \sigma \geq 8$, and assume $\sigma \leq -1$. Then there exists a minimal symplectic 4-manifold $M(G)$ with fundamental group $\pi_1(M(G)) \cong G$, and characteristic numbers
\begin{center}
$c_2(M(G)) = e + 4(g + r)$, and $\sigma(M(G)) = \sigma$.
\end{center}
The intersection form over $\Z$ of $M(G)$ is odd. If the group $G$ is residually finite, then $M(G)$ is irreducible. If the group $G$ is finite, then $M(G)$ admits exotic smooth structures.\\
\end{theorem}

\begin{proof}  The proof of the existence claim is verbatim to the proof of \cite[Theorem 1]{[Y]}. It consists of building the symplectic sum \cite[Theorem 1.3]{[Go]} $M(G): = Z_{1, 1} \#_T BK$ along a symplectic torus of the manifold $BK$ described in Section \ref{Section 2.2} and a manifold $Z_{1, 1}$ from Proposition \ref{Proposition 11}. 

If the group $G$ is residually finite, then $M(G)$ is irreducible by \cite{[HKo]}.

The claim minding the existence of exotic smooth structures on 4-manifolds follows from a pigeonhole argument (cf \cite[Remark 1.2]{[Y]}). An infinite family of pairwise nondiffeomorphic manifolds sharing the same fundamental group and characteristic numbers is constructed using the techniques in \cite{[FS], [FS1], [FPS]}, and keeping track of the Seiberg-Witten invariants using \cite{[Ta], [MMS]}; see \cite[Remark on p. 343]{[BK]}. It is proven in \cite[Corollary 1.5]{[HK]} that there are only finitely many homeomorphism types of closed oriented 4-manifolds with finite fundamental group, and a given Euler characteristic. Thus, there exist exotic smooth structures as it was claimed.
\end{proof}

\subsection{Spin manifolds with negative signature.}{\label{Section 3.2}} In this Section, we prove the second part of Theorem \ref{Theorem 1}. B. D. Park and Z. Szab\'o settled the geography of simply connected spin symplectic 4-manifolds in \cite{[PS]} (see \cite{[JP1]} as well). The task at hand is to extend their results to the geography problem for spin manifolds with finitely presented fundamental group using the building block described in Section \ref{Section 2.2}. We then follow J. Park's work \cite{[JP]} to build exotic smooth structures on the manifolds constructed by using R.Stern and R. Fintushel's Knot surgery \cite{[FS]} and using S. Boyer work \cite{[B]} on homeomorphisms of simply connected 4-manifolds with a given boundary.\\

Before we begin proving Theorem \ref{Theorem 1}, we take a moment to recall S. Boyer's \cite{[B]} results on homeomorphisms of simply connected 4-manifolds with a given boundary. We follow J. Park's exposition in \cite{[JP]}.  Let $V_1$ and $V_2$ be compact connected simply connected 4-manifolds with connected boundary $\partial V_1 \cong \partial V_2$, and let $f: \partial V_1 \rightarrow \partial V_2$ be an orientation-preserving homeomorphism. According to Boyer, there are two obstructions that need be overcome in order to extend $f$ to a homeomorphism $F: V_1\rightarrow V_2$. First, an isometry $\Lambda: (H_2(V_1; \Z), Q_{V_1}) \rightarrow (H_2(V_2; \Z), Q_{V_2})$ for which the following diagram commutes

\begin{center}
$\begin{CD} 
0@>>> H_2(\partial V_1)@>>> H_2(V_1)@>>> H_2(V_1, \partial V_1)@>>>  H_1(\partial V_1)@>>> 0\\ 
@.@VVf_{\ast}V @VV\Lambda V@AA\Lambda^{\ast} A @VVf_{\ast}V\\ 
0@>>> H_2(\partial V_2)@>>> H_2(V_2)@>>> H_2(V_2, \partial V_2)@>>>  H_1(\partial V_2)@>>> 0\\ 
 \end{CD}$ 
\end{center}

must exist. In this context, an isometry is an isomorphism $\Lambda: H_2(V_1; \Z) \rightarrow H_2(V_2; \Z)$ that preserves intersection forms $Q_{V_i}$. The morphism $\Lambda^{\ast}$ in the diagram is the adjoint of $\Lambda$ with respect to the identification of $H_2(V_i, \partial V_i)$ with $Hom(H_2(V_i); \Z)$ that arises from Lefschetz duality. This information is encoded in the pair $(f, \Lambda)$, which is called a \emph{morphism} and it is denoted symbolically by $(f, \Lambda): V_1\rightarrow V_2$.\\

Second, one must find a homeomorphism $F: V_1\rightarrow V_2$ such that the morphism satisfies $(f, \Lambda) = (F|_{\partial V_1}, F_{\ast})$. A morphism $(f, \Lambda)$ is \emph{geometrically realized}, if this condition is satisfied. In \cite{[B]}, S. Boyer studied conditions under which a given morphism can be geometrically realized.

\begin{theorem}{\label{Theorem Bo}} (Boyer \cite[Theorem 0.7, Proposition 0.8]{[B]}). If $(f, \Lambda): V_1\rightarrow V_2$ is a morphism between two simply connected compact smooth 4-manifolds $V_1, V_2$ with boundary $\partial V_1 \cong \partial V_2$, there is an obstruction $\theta(F, \Lambda) \in I^1(\partial V_1)$ such that $(f, \Lambda)$ is realized geometrically if and only if $\theta(f, \Lambda) = 0$. Moreover, if $H_1(\partial V_1; \Q) = 0$, then $\theta(f, \Lambda) = 0$.\\
\end{theorem}

Particularly useful is the following corollary, which appears in J. Park's work.\\

\begin{corollary}{\label{Corollary Bo}} (cf. \cite[Corollary 2.3]{[JP]}). Assume $\partial V_1 \cong \partial V_2$ is a homology 3-sphere. For any morphism $(f, \Lambda): V_1 \rightarrow V_2$ between simply connected smooth 4-manifolds $V_1, V_2$, there is a homeomorphism $F: V_1 \rightarrow V_2$ such that $(f, \Lambda) = (F|_{\partial V_1}, F_{\ast})$.\\ 

\end{corollary}

We are now ready to prove our first result. The precise statement of Theorem \ref{Theorem 1} reads as follows.

\begin{theorem}{\label{Theorem 13}} Let $s\geq 1, n\in \N$, and let $G$ be a finitely presented group. For each of the following pairs of integers
\begin{center}
$(c, \chi) = (8n - 8 + 8(g + r), 2s + n - 1 + (g + r))$,
\end{center}
there exists a symplectic spin 4-manifold $X(G)$ with 
\begin{center}
$\pi_1(X(G)) = G$ and $(c_1^2(X(G)), \chi_h(X(G))) = (c, \chi)$.
\end{center}
If the group $G$ is residually finite, then $X(G)$ is irreducible. Moreover, there is an infinite set $\{X_m(G)): m\in \N\} $ of pairwise nondiffeomorphic symplectic and nonsymplectic manifolds homeomorphic to $X(G)$ for every $m\in \N$.\\
\end{theorem}

The characteristic numbers of the examples of Theorem \ref{Theorem 13} can be stated as  $c_2(X(G)) = 4(n - 4 + 6 s + (g + r))$, and $\sigma(X(G)) = -16 s$.\\

The following proof is a modification to J. Park's work done in \cite{[JP]}.\\

\begin{proof} Let $Z$ be an irreducible simply connected symplectic spin 4-manifold of \cite{[PS]} (see also \cite{[JP1]}.) B. D. Park and Z. Szab\'o constructed $Z$ using simply connected minimal elliptic surfaces without multiple fibers $E(n)$, which are known to contain Gompf nuclei $C_n \subset E(n)$ \cite{[GN]}.  In particular, there exists a symplectic torus $T\hookrightarrow Z$ of self intersection zero, with $\pi_1(Z - T) \cong \{1\}$. Build the symplectic sum \cite[Theorem 1.3]{[Go]}
\begin{center}
$X(G):= Z \#_{T} BK$,
\end{center}

where $BK$ is the manifold described in Section \ref{Section 2.2}. By \cite[Proposition 1.2]{[Go]}, the manifolds $X(G)$ are spin. The computations of the characteristic numbers are straight-forward, since $c_1(T^2) = 0$, $c_2(BK) = 4(g+r)$, and $\sigma(BK) = 0$. The claim $\pi_1(X(G)) \cong G$ follows from the computations in Section \ref{Section 2.3}. If the group $G$ is residually finite, then $M(G)$ is irreducible by \cite{[HKo]}. The existence claim follows from a variation of the choice of symplectic manifold $Z$ one takes from \cite{[PS]} in the construction of the symplectic sum $X(G)$. 

From the existence of Gompf nuclei inside the building block $Z$, one concludes that there exists a symplectic torus $T' \hookrightarrow X(G)$ such that the inclusion induces the trivial map $\pi_1(T') \rightarrow \pi_1(X(G) - T') \cong G$. The infinite set of pairwise nondiffeomorphic symplectic and nonsymplectic 4-manifolds $\{X_m(G) : m \in \N\}$ is constructed using R. Fintushel - R. Stern's knot surgery \cite{[FS]},  varying the choice of knot.

To prove the claim on the homeomorphism type of the manifolds constructed, we use an argument due to J. Park \cite[Proposition 2.2]{[JP]}. Let the 4-manifold with arbitrary fundamental group $X_i(G):= Z_{K_i} \#_T BK \in \{X_m(G): m \in \N\}$ be the result of applying knot surgery with knot $K_i$. We claim that there exists a homeomorphism between the manifolds $X_i(G) = Z_{K_i} \#_T BK$ and $X_j(G) = Z_{K_j} \#_T BK$ for any $i, j \in \N$. In this notation, we are specifying that the surgery occurs in the $Z$ block of the construction, and calling $Z_{K_i}$ the homeomorphic simply connected manifolds produced using knot surgery with knot $K_i$. The spin manifold $Z_{K_i}$ contains a Gompf nucleus as a codimensional zero submanifold, independent of the knot $K_i$. Let $C_n$ be the Gompf nucleus, let $\Sigma$ be a homology 3-sphere,  and set $Z_{K_i}^{\circ}:= Z_{K_i} - C_n$. The symplectic sum  $Z_{K_i} \#_T BK$ can be decomposed as
\begin{center}
$Z_{K_i}^{\circ}\cup_{\Sigma} (BK\#_T C_n)  $.
\end{center}

The endeavor at hand is to construct a morphism $(f, \Lambda)$, and then use Boyer's results to show that it is geometrically realized. The choice of homeomorphism $f$ is canonical by setting 
$g:= id: BK \#_T C_n \rightarrow BK \#_T C_n$, so that its restriction yields $g|_{\partial(BK \#_T C_n)} = f: \Sigma \rightarrow \Sigma$. Take an isomorphism $\Lambda: H_2(Z_{K_i}^{\circ}; \Z)\rightarrow H_2(Z_{K_j}^{\circ}; \Z)$. Minding the intersection forms, $Q_{Z_{K_i}^{\circ}}$ and $Q_{Z_{K_i}^{\circ}}$ are unimodular, indefinite forms over $\Z$ that have the same rank, same signature, and same type. By the classification of unimodular indefinite integral forms,  $Q_{Z_{K_i}^{\circ}} \cong Q_{Z_{K_i}^{\circ}}$. We have that $\Lambda:(H_2(Z_{K_i}^{\circ}; \Z), Q_{Z_{K_i}^{\circ}}) \rightarrow (H_2(Z_{K_j}^{\circ}; \Z), Q_{Z_{K_j}^{\circ}})$ is an isomorphism that preserves intersection form, i.e., an isometry. Therefore, we have a required morphism $(f, \Lambda): Z_{K_i}^{\circ}\rightarrow Z_{K_j}^{\circ}$.

Since $H_1(\partial(Z_{K_i}^{\circ}); \Q) = H_1(\partial(Z_{K_i}^{\circ}); \Q) = H_1(\Sigma; \Q) = 0$, then $\theta(f, \Lambda) = 0$ by Theorem \ref{Theorem Bo}. It follows from Corollary \ref{Corollary Bo} that there exists a homeomorphism $F: Z_{K_i}^{\circ} \rightarrow Z_{K_j}^{\circ}$ such that $(f, \Lambda) = (F|_{\Sigma}, F_{\ast})$. Thus, for $i, j \in \N$ any two manifolds $X_i(G) = Z_{K_i}\#_T BK$ and $X_j(G) = Z_{K_j} \#_T BK$ are homeomorphic.
\end{proof}

\begin{remark}{\label{Remark 15}} \emph{On the geography for other choices of $\pi_1$ and other regions.} The techniques used in this paper yield similar results to Theorem \ref{Theorem 1} for several choices of fundamental groups. For example, surface groups,  3-manifolds groups, knot group of (p, q)-torus knots, the 3-dimensional Heisenberg group, the infinite dihedral group, $PSL(2, \Z)$, free groups of rank $n$, amongst others cf. \cite{[BK1]}, and Example \ref{Example E} below.
Moreover, other regions can be populated using similar methods cf. \cite{[BK2]}\\
\end{remark}

\section{Geography of symplectic 6-manifolds.}{\label{Section 4}}

\subsection{Computing Chern numbers.}{\label{Section 4.1}}

Minding the 6-manifolds arising as products of 4- and 2-manifolds, we have the following lemma.

\begin{lemma}{\label{Lemma P}} Let $Y$ be a symplectic 4-manifold, and $\Sigma_g$ a symplectic surface of genus $g$. Let $e:=e(Y)$ be the Euler characteristic and $\sigma:= \sigma(Y)$ the signature.The Chern numbers of the product symplectic 6-manifold $X:= Y\times \Sigma_g$ are given as follows.
\begin{itemize}
\item $c_1^3(X) = (2e + 3\sigma) \cdot (6 - 6g)$,
\item $c_1 c_2(X) =  (e + \sigma) \cdot (6 - 6g)$, and
\item $c_3(X) = e \cdot (2 - 2g)$.
\end{itemize}
\end{lemma}

\begin{proof} We need to identify the Chern classes. For such a purpose, consider the projections
\begin{center}
$\pi_1: X = Y\times \Sigma_g \rightarrow Y$,
$\pi_2: X = Y\times \Sigma_g \rightarrow \Sigma_g$.
\end{center}

The total Chern class can be written as $c(X) = \pi_1^*(c(Y))\cdot \pi_2^*( c(\Sigma_g))$. The lemma follows by integrating over $X = Y\times \Sigma_g$.\\
\end{proof}

The change on the Chern numbers of 6-manifolds after blow ups at points and along surfaces are given by the following formulas. 

\begin{lemma}{\label{Lemma B}} Let $X$ be a symplectic 6-manifold. Let $X'$ be the symplectic 6-manifold obtained by blowing up a point in $X$. The Chern numbers of $X'$ are as follows: \begin{itemize}
\item $c_1^3(X') = c_1^3(X) - 8$,
\item $c_1c_2(X') = c_1c_2(X)$,
\item $c_3(X') = c_3(X) + 2$.
\end{itemize}

Let $\Sigma_g \hookrightarrow X$ be a symplectic surface of genus $g$, and denote by $N(\Sigma_g)$ its normal bundle. Let $X'_g$ be the symplectic 6-manifold obtained by blowing up $X$ along $\Sigma_g$. The Chern numbers of $X'_g$ are as follows:
\begin{itemize}
\item $c_1^3(X'_g) = c_1^3(X) + 6(g - 1) - 2 \left<c_1(N(\Sigma_g)), [\Sigma_g]\right>$,
\item $c_1c_2(X'_g) = c_1c_2(X)$,
\item $c_3(X'_g) = c_3(X) - 2(g - 1)$.
\end{itemize}

\end{lemma}

See \cite[Lemma 2.1]{[MH]} for a proof.

Regarding the 6-manifolds obtained as symplectic sums, we have the following lemma.

\begin{lemma}{\label{Lemma S}} Let $X_1, X_2$ be symplectic 6-manifolds that contain 4-dimensional symplectic submanifolds $Y_i \subset X_i$ for $i = 1, 2$ with trivial normal bundle such that $Y_1 \cong Y_2$. Set $Y:= Y_i$. The Chern numbers of the symplectic sum $X:= X_1 \#_{Y} X_2$ are
\begin{center}
$c_1^3(X) = c_1^3(X_1) + c_1^3(X_2) - 6 c_1^2(Y)$,\\
$c_1c_2(X) = c_1c_2(X_1) + c_1c_2(X_2) - 2(c_1^2(Y) + c_2(Y))$, and\\
$c_3(X) = c_3(X_1) + c_3(X_2) - 2 c_2(Y)$.\\

\end{center}
\end{lemma}

See \cite[Proposition 1.3]{[MH]} for a proof (cf. \cite[Lemma 3.23]{[BP]}).\\

The existence of a a symplectic structure induces an almost-complex structure on the underlying manifold. The possible restrictions on the Chern numbers of almost-complex 6-manifolds are computed in \cite[Proposition 9]{[OV]}, \cite{[HG]}, and see \cite{[CL]} for interesting 6-dimensional phenomena.

\subsection{6-dimensional symplectic sums}{\label{Section 4.2}} We use our 4-dimensional building blocks to construct symplectic 6-manifolds using Gompf's symplectic sum construction \cite[Theorem 1.3]{[Go]} (see also \cite{[MT]}) as our main manufacture tool. In our notation,  $X \#_{Z} Y$ denotes the symplectic sum of $X$ and $Y$ along the submanifold $Z$. 


\begin{theorem}{\label{Theorem 19}} (cf. Theorem \ref{Theorem 2}). Let $G \in \{\Z_p, \Z_p \oplus \Z_q, \Z\oplus \Z_q, \Z, \pi_1(\Sigma_g), F_n\}$, where $\Sigma_g$ is a closed oriented surface and $F_n$ is the free group of rank $n$. Let $e_i, \sigma_i$ $i =1, 2$ be integers that satisfy $2e_i + 3 \sigma_i \geq 0, e_i + \sigma_i \equiv 0$ mod 4, and $e_i + \sigma_i \geq 8$. There exist symplectic 6-manifolds $W_0(G), W_1(G), W_2(G)$, with fundamental group $\pi_1(W_i(G)) \cong G$ for $i = 0, 1, 2$,  and the following Chern numbers:
\begin{enumerate}

\item 
$c_1^3(W_0(G)) = 6 \cdot (3\sigma_1 + 2 e_1 + 3\sigma_2 + 2 e_2)$,\\
$c_1c_2(W_0(G)) = 6 \cdot (e_1 + \sigma_1 + e_2 + \sigma_2)$, and\\
$c_3(W_0) = 2\cdot (e_1 + e_2)$;\\

\item 
$c_1^3(W_1(G)) = 0$,\\
$c_1c_2(W_1(G)) = 0$, and\\
$c_3(W_1(G)) = 0$;\\

\item 
$c_1^3(W_2(G)) =  - 6 \cdot (3\sigma_1 + 2 e_1 + 3\sigma_2 + 2 e_2) - 48$,\\
$c_1c_2(W_2(G)) =  - 6 \cdot (e_1 + \sigma_1 + e_2 + \sigma_2) - 24$, and\\
$c_3(W_2(G)) =  - 2\cdot (e_1 + e_2) - 8$.\\
\end{enumerate}

Assume $s_1, s_2\geq 1$, and $n_1, n_2\in \N$. There exist spin symplectic 6-manifolds $Y_i(G)$ for $i = 0, 1, 2$ with fundamental group $\pi_1(Y_i(G)) \cong G$, and whose Chern numbers are\\
\begin{enumerate}
\item 
$c_1^3(Y_0(G)) = 48 \cdot (n_1 + n_2 - 2)$,\\
$c_1c_2(Y_0(G)) = 48 \cdot (s_1 + s_2) + 24\cdot (n_1 + n_2 - 2)$, and\\
$c_3(Y_0(G)) = 48\cdot (s_1 + s_2) + 8 \cdot (n_1 + n_2 - 2)$;\\

\item 
$c_1^3(Y_1(G)) = 0$,\\
$c_1c_2(Y_1(G)) = 0$, and\\
$c_3(Y_1(G)) = 0$;\\

\item 
$c_1^3(Y_2(G)) =  - 48 \cdot (n_1 + n_2 - 1)$,\\
$c_1c_2(Y_2(G)) =  - 48 \cdot (s_1 + s_2) - 24 \cdot (n_1 + n_2 - 1)$, and\\
$c_3(Y_2(G)) =  - 48\cdot (s_1 + s_2) - 8(n_1 + n_2 - 1)$.\\
\end{enumerate}

\end{theorem}

\begin{proof} Let us construct the manifolds $W_i(G)$ for $i = 0, 1, 2$; the method of construction of the manifolds $Y_i(G)$ is similar. Let $X_1, X_2$ be minimal symplectic 4-manifolds of Proposition \ref{Proposition 11}. Moreover, using the procedure described in Section \ref{Section 5.2} on $X_1$, we can produce a minimal symplectic 4-manifold $X_1(G)$ such that $\pi_1(X_1(G)) \cong G$. The manifold $X_1(G)$ contains a homologically essential Lagrangian torus $T:= T_1$ such that $\pi_1(X_1(G) -T) \cong G$. Perturb the symplectic sum on $X_1(G)$ so that $T$ becomes symplectic \cite[Lemma 1.6]{[Go]}. We make use of the following prototypes of symplectic sums according to the item in the statement of the theorem.

\begin{center} (1): $W_0(G):= X_1(G) \times S^2 \#_{T^2\times S^2} X_2 \times S^2$\end{center}

 \begin{center} (2): $W_1(G):= X_1(G) \times T^2 \#_{T^2\times T^2} X_2 \times T^2$\end{center}

\begin{center} (3): $W_2(G): = X_1(G) \times \Sigma_2 \#_{\Sigma_2\times \Sigma_2} X_2 \times \Sigma_2$.\end{center}

The reader should notice that the manifolds $X_1, X_2$ are to be chosen from the manifolds $Z_{1, 1}, Z_{1, 2}$ of Proposition \ref{Proposition 11} appropriately, so that they contain the necessary symplectic submanifolds for the symplectic sums performed. 

A description of the choices of gluing maps for the symplectic sums is given in Section \ref{Section 5.1}. The computations of Chern numbers are straight-forward, and follow from Lemma \ref{Lemma P} and Lemma \ref{Lemma S}. Minding the claim about fundamental groups, we refer the reader to  Section \ref{Section 5}. The manifolds $Y_i(G)$ are constructed in a similar way by taking the manifolds $X_1, X_2$ in the procedure described above from \cite{[PS]}. The details are left to the reader.\\
\end{proof}

\subsection{Realization of all possible combinations of Chern numbers.}{\label{Section 4.3}}

The geography of \emph{non-minimal} simply connected symplectic 6-manifolds was settled by M. Halic in \cite{[MH]}. The examples given by him rely heavily on the blow-ups of some basic building blocks. His main tool to realize all possible combinations of Chern numbers with a symplectic 6-manifold is stated in the following result.

\begin{lemma}{\label{Lemma M}} M. Halic \cite[Lemma 4.1]{[MH]}. Let $X$ be a symplectic 6-manifold whose Chern numbers are given by $(c_1^3, c_1c_2, c_3) = (2a, 24 b, 2 c)$, and which has the following properties. We denote by $N(E_X)$ the normal bundle of a projective line $E \subset X$.
\begin{itemize}
\item The manifold $X$ contains a symplectically embedded product $U\times D$, where $U$ is an open subset of a symplectic 4-manifold and $D$ is a disc. Assume that $U$ contains a projective line $E$ having the property that \begin{center}$- \alpha:= \left<c_1(N(E_{X})), [E]\right> \leq - 1$.\end{center}

\item There is a symplectic genus two surface $\Sigma_2 \subset X$, which is disjoint from the exceptional sphere $E$, and $[\Sigma_2]^2 = 0$.
\end{itemize}

Then, by blowing up $p$ points, $r$ distinct copies of $E$, and $z$ distinct copies of $\Sigma_2$, one obtain all triples of Chern numbers of the form $(2 a', 24 b, 2 c')$, where $a', c'$ are arbitrary integers.
\end{lemma}


Theorem \ref{Theorem 19} allows us to recover the main result of \cite{[MH]}, and extend it to manifolds with other choices of fundamental groups as exemplified in the following statement (cf. Remark \ref{Remark 15} and Corollary \ref{Corollary 3}). 

\begin{example}{\label{Example E}} Let $X(G)$ be a symplectic 6-manifold. Then, 
\begin{center}
$c_1^3(X(G)) \equiv c_3(X(G)) \equiv 0$ mod 2, and\end{center} \begin{center} $c_1c_2(X(G)) \equiv 0 $ mod 24.
\end{center}
Conversely, let  $G \in \{\Z_p, \Z_p \oplus \Z_q, \Z\oplus \Z_q, \Z, \pi_1(\Sigma_g), F_n\}$. Any triple $(a, b, c)$ of integers satisfying $a \equiv c \equiv 0$ mod 2 and $b \equiv 0$ mod 24 occurs as a triple $(c_1^3, c_1c_2, c_3)$ of Chern numbers of a symplectic 6-manifold $X(G)$ with $\pi_1(X(G)) \cong G$. 
\end{example}

\begin{proof} Let $X$ be a symplectic 6-manifold. Then $X$ admits an almost-complex structure, and \cite[Proposition 9]{[OV]} implies that the Chern numbers of $X$ satisfy $c_1^3(X) = c_3(X) \equiv 0$ mod 2 and $c_1c_2(X) \equiv 0$ mod 24.

Now, let $X_1(G)$ be the manifold as in the proof of Theorem \ref{Theorem 19}. There are two things to be checked in order to establish the converse. First, one needs to see that the result of blowing up at a point the manifolds constructed in Theorem \ref{Theorem 19}  satisfy the hypothesis of Lemma \ref{Lemma M}. As a consequence, we are then able to conclude that all the possible values of $c_1^3$ and $c_3$ are realized by a symplectic 6-manifold with the given fundamental group. Second, it needs to be seen that such manifolds fulfill all the possible values of $c_1c_2$ (values of $b$ in the statement of Lemma \ref{Lemma M}). Notice that by Lemma \ref{Lemma B}, the value $c_1c_2$ is invariant under blow ups of points and under blow ups along surfaces.

By Proposition \ref{Proposition 11}, there exists a symplectic surface of genus two $F$ inside the submanifold $X_1(G) \times \{x\} \subset X_1(G) \times \Sigma_g$ for every 6-manifold obtained as a symplectic sum in Theorem \ref{Theorem 19}. By \cite{[Go]}, every manifold of Theorem \ref{Theorem 19} contains a symplectic surface of genus two. The manifold obtained by blowing up a point that is disjoint from the surface of genus two satisfies the hypothesis of Lemma \ref{Lemma M}.

We now consider the possible values of $c_1c_2$. The claim is that all possible values are obtained by the manifolds constructed in Theorem \ref{Theorem 19}. The case $c_1c_2 = 0$ is clear. The claim is equivalent to finding  $e_i, \sigma_i, e_i', \sigma_i'$ $i = 1, 2$ that produce integers $k_+$ and $k_-$ that satisfy
\begin{center}
$6\cdot (e_1 + \sigma_1 + e_2 + \sigma_2) = 24k_+$ and 
$-6\cdot (e_1' + \sigma_1' + e_2' + \sigma_2') = 24k_-$.\\
\end{center}

The numbers $e_i, \sigma_i, e_i', \sigma_i'$ $i = 1, 2$ are as in the statement of Theorem \ref{Theorem 19}, and they correspond to the Euler characteristic and signatures of minimal 4-manifolds of Proposition \ref{Proposition 11}, where $X_1(G)$ is constructed from. The corresponding manifolds to these integer numbers can be chosen so that the equalities $(e_1 + \sigma_1 + e_2 + \sigma_2) =  4k_+$ and $(e_1' + \sigma_1' + e_2' + \sigma_2') = 4k_-$ are satisfied. Thus, all possible values of $c_1c_2$ are realized by manifolds from Theorem \ref{Theorem 19}. The corollary now follows from Lemma \ref{Lemma M}.
\end{proof}

\section{Fundamental group computations}{\label{Section 5}}

\subsection{Choices of gluing maps for symplectic sums in Section \ref{Section 4.2}}{\label{Section 5.1}} An important piece of data in the 6-dimensional symplectic sum constructions of Section \ref{Section 4.2} are the gluing maps. Here we proceed to describe them in detail in order to conclude the fundamental group computations claimed in the statements of Theorem \ref{Theorem 2}, \ref{Corollary 3}, and Theorem \ref{Theorem 4}. The setting is as follows. Let $X$ and $Z$ be symplectic 4-manifolds containing symplectic tori $T\subset X$, and $T'\subset Z$ or symplectic surfaces of genus two $F\subset X$ and $F'\subset Z$ (see Proposition \ref{Proposition 11}); notice that according to Proposition \ref{Proposition 11} or Theorem \ref{Theorem Y} one might need to perturb the symplectic form of the ambient 4-manifold using \cite[Lemma 1.6]{[Go]} for the homologically essential Lagrangian tori to be assumed symplectic. \\

Regarding the homotopy  computations, we assume \begin{center}$\pi_1(X - T) \cong \pi_1(X)$ and $\pi_1(Z - T') \cong \pi_1(Z) \cong \{1\}$,\end{center}
\begin{center}$\pi_1(X - F) \cong \pi_1(X)$ and $\pi_1(Z - F') \cong \pi_1(Z) \cong \{1\}$.\end{center}

Take a symplectic surface of genus $g$, and build the 6-manifolds $X \times \Sigma_g$ and $Z\times \Sigma_g$ equipped with the product symplectic form. For $g = 0, 1$, these symplectic 6-manifolds contain a symplectic submanifold $T^2\times \Sigma_g$: $T\times \Sigma_g \subset X\times \Sigma_g$ and $T'\times \Sigma_g \subset Z\times \Sigma_g$. Similarly, we have the symplectic submanifolds (case $g = 2$) $F\times \Sigma_2 \subset X\times \Sigma_2$ and $F'\times \Sigma_2 \subset Z\times \Sigma_2$.  In all our constructions, $\Sigma_g$ is either a 2-sphere $S^2$ ($g = 0$), a 2-torus $T^2$ ($g = 1$), or  a surface of genus two $\Sigma_2$ ($g = 2$); surfaces of higher genus can be used with a small modification to our arguments.\\

For surfaces of genus $g = 0, 1$, the 6-manifolds are symplectic sums of the form

\begin{center} $W_g:= X \times \Sigma_g \#_{T\times \Sigma_g = T'\times \Sigma_g} Z \times \Sigma_g$,
\end{center}

where the gluing map is a diffeomorphism \begin{center}$\phi_g: \partial(T\times \Sigma_g \times D^2) \rightarrow \partial(T'\times \Sigma_g \times D^2) \cong T^2\times \Sigma_g \times S^1$.\\ \end{center}

Similarly, for surfaces of genus two we built 

\begin{center} $W_2:= X \times \Sigma_2 \#_{F\times \Sigma_2 = F'\times \Sigma_2} Z \times \Sigma_2$,
\end{center}

where the gluing map given by a diffeomorphism \begin{center}$\phi_2: \partial(F\times \Sigma_2\times D^2) \rightarrow \partial(F'\times \Sigma_2 \times D^2) \cong \Sigma_2\times \Sigma_2 \times S^1$.\\ \end{center}

\begin{lemma}{\label{Lemma F}} $\pi_1(W_g) \cong \pi_1(X)$ for $g = 0, 1, 2$.

\end{lemma}

\begin{proof}
We use Seifert-van Kampen Theorem to compute the fundamental group of $W_g$ for $g = 0, 1, 2$. In order to obtain $\pi_1(W_g) \cong \pi_1(X)$, we choose a diffeomorphism $\phi_g$ that maps the generators of the fundamental groups involved as follows; remember we can assume $\pi_1(Z - T') \cong \pi_1(Z) \cong \{1\}$ (see Proposition \ref{Proposition 11}). The case $g = 0$ is clear, since there are isomorphisms $\pi_1(X\times S^2 - T\times S^2) \cong \pi_1(X)$ and $\pi_1(Z\times S^2 - T'\times S^2) \cong \{1\}$.\\

We work out the case $g = 1$; the case $g = 2$ follows verbatim from our argument. In this case, the symplectic 4-manifold along which the gluing in the symplectic sum \cite{[Go]} is performed is the 4-torus

\begin{center}
$T\times T^2 \cong T^4 \cong T'\times T^2$.\\
\end{center}

Let the groups $\pi_1(X\times T^2 - T\times T^2) \cong \pi_1(X) \oplus \Z\oplus \Z$ and $\pi_1(Z\times T^2 - T'\times T^2)\cong \Z \oplus \Z$ be generated by $\{x_1, x_2, \ldots, x_n, x = 1, y = 1, w, z\}$ and $\{x' = 1, y' = 1, w', z'\}$ respectively, where $\{x_1, x_2, \ldots, x_n\}$ are generators of $\pi_1(X)$, $\{x, y\}$ are push offs of generators of $\pi_1(T)$, and each element $\{w', z'\}$ generates a $\Z$. Let $\mu$ and $\mu'$ be the respective (based) meridians of the submanifolds. Notice that in both cases the meridians $\mu = 1 = \mu'$ are trivial; we assume $\mu$ gets identified with $\mu'^{-1}$ during the gluing employed to construct $W_1$. Choose $\phi_1$ to be the diffeomorphism that identifies the generators of the fundamental groups as follows:
\begin{center}
$x\mapsto w', y\mapsto z', w\mapsto x', z\mapsto y'$.
\end{center}

This choice of diffeomorphism allows us to conclude that the extra generators $w, z, w', z'$ are killed during the gluing. An application of the Seifert-van Kampen Theorem implies $\pi_1(W_1) \cong \pi_1(X)$, as desired.
\end{proof}

\subsection{Computations of $\pi_1$ Theorem \ref{Theorem 2}, Theorem \ref{Theorem 4}, Corollary \ref{Corollary 3}, and Theorem \ref{Theorem 19}.} {\label{Section 5.2}} Minding the geography of symplectic 6-manifolds, the desired fundamental groups are obtained as follows. A minimal 4-manifold $Z_{1, 1}$ of Proposition \ref{Proposition 11} contains two homologically essential Lagrangian tori with trivial meridians; in particular, there is an isomorphism $\pi_1(Z_{1,1} - (T_1 \cup T_2)) \rightarrow \pi_1(Z_{1,1}) \cong \{1\}$ (similar argument works for $Z_{1, 2}$).\\

The idea is to build the symplectic sum of $Z_{1,1}$ along $T_1$ and a chosen manifold, in order to obtain a manifold $P$ with the desired fundamental group $G$. The homologically essential Lagrangian torus $T_2 \subset P$ satisfies  that the homomorphism induced by inclusion $T_2\hookrightarrow P$ is an isomorphism $\pi_1(P - T_2) \rightarrow \pi_1(P) \cong G$. This is then used in our 6-dimensional symplectic sums in Section \ref{Section 4.2} to obtain the prescribed fundamental group. We proceed to prove the fundamental group claims on the four-dimensional building blocks employed.

\begin{proof} \begin{itemize} \item Case abelian group of small rank. To be precise, we construct manifolds with fundamental group among $\{\{1\}, \Z_p, \Z_p \oplus \Z_q, \Z, \Z \oplus \Z_q, \Z\oplus \Z\}$. Consider the 4-torus $T^2\times T^2$ endowed with the product symplectic form. The 2-torus $\{x\}\times T^2 \subset T^2\times T^2$ is a symplectic submanifold of self intersection zero. Start by building the symplectic sum
\begin{center}
$Z(\Z\oplus \Z):= {Z_{1,1}}_{T_1 = T^2} T^2\times T^2$.
\end{center}

Since $\pi_1(Z_{1,1}) \cong \pi_1(Z_{1,1} - T_1) \cong \{1\}$, it follows from Seifert-van Kampen Theorem that $\pi_1(Z(\Z\oplus\Z)) \cong \Z\oplus \Z$.  There are two pairs of geometrically dual Lagrangian tori inside $Z(\Z\oplus \Z)$, on which one can perform Luttinger surgeries \cite{[Lu], [ADK]}. These tori come from the Lagrangian tori inside the 4-torus used in the symplectic sum construction. Applying the appropriate surgery yields the desired fundamental groups: see  \cite[Section 2]{[BK0]}, \cite[Theorem 1]{[BK]} for the fundamental group computations, and \cite{[T0]} for details.

\item Case surface groups. Manifolds with fundamental group $\{1\}$ and $\Z\oplus \Z$ were constructed in the previous item. To construct a manifold $S$ with $\pi_1(S) \cong \pi_1(\Sigma_g)$ for $g\geq 2$, we proceed as follows. Consider the product $T^2\times \Sigma_g$ of a 2-torus and a genus $g\geq 2$ surface endowed with the product symplectic form. The submanifold $T^2\times \{x\} \subset T^2\times \Sigma_g$ is symplectic and its self intersection number is zero. One builds the symplectic sum 
\begin{center}
$S:= {Z _{1,1}}_{T_1 = T^2} T^2 \times \Sigma_g$.
\end{center}

Since $\pi_1(Z_{1,1}) \cong \pi_1(Z_{1,1} - T_1) \cong \{1\}$, it follows from Seifert-van Kampen Theorem that $\pi_1(S) \cong \pi_1(\Sigma_g)$ as desired.
\item Case free groups of rank $n \in \N$, $F_n$. The case $n = 1$ was proven previously. Let $\Sigma_{n}$ be a surface of genus $n$, and denote the standard generators of $\pi_1(\Sigma_{n})$ by $x_1, y_1, \ldots, x_{n}, y_{n}$. Let $\phi: \Sigma_{n} \rightarrow \Sigma_{n}$ be a diffeomorphism consisting of the composition of $n - 1$ Dehn twists along the loops $x_i$; denote by $Y$ the mapping torus of $\phi$ that  fibers over the circle.

Build the product  $Y\times S^1$, which admits a symplectic structure \cite{[T]}. This manifold contains a symplectic torus $T'$,  given by $T' = S\times S^1$, where $S$ is a section of $Y \rightarrow S^1$. The fundamental group $\pi_1(Y \times S^1)$ is calculated by expressing $\pi_1(Y)$ as an HNN extension, which yields
\begin{center}
$\pi_1(Y\times S^1) = \left<x_i, y_i, t : tx_i t^{-1} = x_i, ty_it^{-1} = y_i x_i\right> \oplus \Z s$.\\
\end{center}

Build the symplectic sum
\begin{center}
$F(n):= {Z_{1,1}}_{T_1 = T'} Y\times S^1$.\\
\end{center}

Since $\pi_1(Z_{1,1}) \cong \pi_1(Z_{1,1} - T_1) \cong \{1\}$, it follows from Seifert-van Kampen Theorem that the generators $t$ and $s$ are killed during the gluing. The relations $t = 1 = s$ imply that $x_i = 1$ for all i. The Seifert-van Kampen Theorem implies that the group $\pi_1(F(n))$ is generated by the $y_i$'s with no relations among them in the group presentation. That is, $\pi_1(F(n))$ is a free group on $n$ generators. 

\end{itemize}
\end{proof}

\begin{remark}\emph{On the minimality of  6-manifolds}. The notion of minimality for symplectic 6-manifolds (see \cite[Definition 5.2]{[TY]}) in generality is yet to be established  (compare with the notion in dimension four \cite{[GS]}); this is an interesting problem (cf. \cite{[TY]}).\\
\end{remark}

\subsection{Proofs of Theorem \ref{Theorem 4} and Corollary \ref{Corollary 3}.}{\label{Section 5.3}} The paper finishes with the proofs of the results promised in the Introduction \ref{Intro}. The precise statements are the following. 


\begin{theorem} (cf. Theorem \ref{Theorem 4}).{\label{Theorem 22}}  Let $G$ be a group with a presentation that consists of $g$ generators and $r$ relations. Let $e_i, \sigma_i$ $i =1, 2$ be integers that satisfy $2e_i + 3 \sigma_i \geq 0, e_i + \sigma_i \equiv 0$ mod 4. There exist symplectic 6-manifolds $W_0(G), W_1(G), W_2(G)$ with the following Chern numbers:
\begin{enumerate}

\item 
$c_1^3(W_0(G)) = 6 \cdot (3\sigma_1 + 2( e_1 + 4(g + r)) + 3\sigma_2 + 2 e_2)$,\\
$c_1c_2(W_0(G)) = 6 \cdot ((e_1 + 4(g + r)) + \sigma_1 + e_2 + \sigma_2)$, and\\
$c_3(W_0) = 2\cdot ((e_1 + 4(g + r)) + e_2)$;\\

\item 
$c_1^3(W_1(G)) = 0$,\\
$c_1c_2(W_1(G)) = 0$, and\\
$c_3(W_1(G)) = 0$;\\

\item 
$c_1^3(W_2(G)) =  - 6 \cdot (3\sigma_1 + 2 (e_1 + 4(g + r)) + 3\sigma_2 + 2 e_2) - 48$,\\
$c_1c_2(W_2(G)) =  - 6 \cdot ((e_1 + 4(g + r)) + \sigma_1 + e_2 + \sigma_2) - 24$, and\\
$c_3(W_2(G)) =  - 2\cdot ((e_1 + 4(g + r)) + e_2) - 8$.
\end{enumerate}
There is an isomorphism $\pi_1(W_i) \cong G$ (for $i = 0, 1, 2$).\\

Assume $s_1, s_2\geq 1$, and $n_1, n_2\in \N$. There exist spin symplectic 6-manifolds $Y_i(G)$ for $i = 0, 1, 2$ with $\pi_1(Y(G)) \cong G$, and whose Chern numbers are
\begin{enumerate}
\item 
$c_1^3(Y_0(G)) = 6 \cdot (8n_1 - 8 + 8(g + r) + 8n_2 - 8)$,\\
$c_1c_2(Y_0(G)) = 6 \cdot (4n_1 - 4 + 4(g + r) + 8s_1 + 4n_2 - 4 + 8s_2)$, and\\
$c_3(Y_0(G)) = 2\cdot (4n_1 - 4 + 4(g + r) + 24s_1 + 4n_2 - 4 + 24s_2)$;\\

\item 
$c_1^3(Y_1(G)) = 0$,\\
$c_1c_2(Y_1(G)) = 0$, and\\
$c_3(Y_1(G)) = 0$;\\

\item 
$c_1^3(Y_2(G)) =  - 6 \cdot (8n_1 - 8 + 8(g + r) + 8n_2 - 8) - 48$,\\
$c_1c_2(Y_2(G)) =  - 6 \cdot (4n_1 - 4 + 4(g + r) + 8s_1 + 4n_2 - 4 + 8s_2) - 24$, and\\
$c_3(Y_2(G)) =  - 2\cdot (4n_1 - 4 + 4(g + r) + 24s_1 + 4n_2 - 4 + 24s_2) - 8$.
\end{enumerate}

\end{theorem} 

\begin{proof} The argument is verbatim to the one employed in proof of Theorem \ref{Theorem 19}, modulo one of the building blocks. This new piece is going to be taken from the construction of Theorem \ref{Theorem Y}. The changes go as follows. Take a minimal symplectic 4-manifold $X_1$ from Proposition \ref{Proposition 11}, and perturb the symplectic form so that the torus $T_1$ becomes symplectic \cite[Lemma 1.6]{[Go]}. Build the 4-dimensional symplectic sum \cite{[Go]}

\begin{center}

$X_1(G):= BK \#_{T = T_1} X_1$

\end{center}

where $BK$ is the minimal symplectic 4-manifold with Euler characteristic given by $c_2(BK) = 4(g + r)$, signature $  \sigma(BK) = 0$, and fundamental group $\pi_1(BK) \cong \left<g_1, \ldots, g_g | r_1, \ldots, r_r\right> \cong G$ described in Section \ref{Section 2.2}. A straight-forward computation yields $c_2(X_1(G)) = c_2(X_1) + 4(g + r), \sigma(X_1(G)) = \sigma(X_1)$. The discussion in Section \ref{Section 2.3} implies $\pi_1(X_1(G)) \cong G$, which explains our choice of notation.

We proceed to build the 6-dimensional examples. Let $X_2$ be a minimal symplectic 4-manifold of Proposition \ref{Proposition 11}. We make use of the following prototypes of symplectic sums according to the item in the statement of the theorem.

\begin{center} (1): $W_0(G):= X_1(G) \times S^2 \#_{T^2\times S^2} X_2 \times S^2$\end{center}

 \begin{center} (2): $W_1(G):= X_1(G) \times T^2 \#_{T^2\times T^2} X_2 \times T^2$\end{center}

\begin{center} (3): $W_2(G): = X_1(G) \times \Sigma_2 \#_{\Sigma_2\times \Sigma_2} X_2 \times \Sigma_2$.\end{center}

A detailed description of the gluing maps chosen for the symplectic sums is given in Section \ref{Section 5.1}. Varying our choices of manifolds $X_1$ and $X_2$ results in the realization of the Chern numbers claimed; under our notation $c_2(X_i) = e_i, \sigma(X_i) = \sigma_i$ for $i = 1, 2$. The computations of the Chern numbers are straight-forward, and follow from Lemma \ref{Lemma P} and Lemma \ref{Lemma S}.

The construction of the manifolds $Y_i(G)$ for $i = 0, 1, 2$ follows the same mechanism described above, using as building blocks the manifolds of Theorem \ref{Theorem 13} and the reference \cite{[PS]}. 
\end{proof}


We prove Corollary  \ref{Corollary  3} (cf. Proof of Example \ref{Example E}).

\begin{corollary}{\label{Corollary C}}
(cf. \cite[Proposition 4.31]{[BP]}). Let $(a, b, c)$ be any triple of integers, and $G$ a finitely presented group. There exists a symplectic 6-manifold $X(G)$ with $\pi_1(X(G)) \cong G$, and whose Chern numbers are
\begin{center}
$(c_1^3(X(G)), c_1c_2(X(G)), c_3(X(G))) = (2 a, 24b, 2 c)$.
\end{center}
\end{corollary}

\begin{proof}  
Let $X(G):= X_1(G)$ be as in the proof of Theorem \ref{Theorem 22}.  There are two things to be checked. First, one needs to see that the result of blowing up at a point the manifolds constructed in Theorem \ref{Theorem 19} satisfy the hypothesis of Lemma \ref{Lemma M}. As a consequence, we are then able to conclude that all the possible values of $c_1^3$ and $c_3$ are realized by a symplectic 6-manifold with the given fundamental group. Second, it is needed to be seen that such manifolds fulfill all the possible values of $c_1c_2$ (values of $b$ in the statement of Lemma \ref{Lemma M}). Notice that by Lemma \ref{Lemma B}, the value $c_1c_2$ is invariant under blow ups of points and under blow ups along surfaces.

By Proposition \ref{Proposition 11}, there exists a symplectic surface of genus two $F$ inside the submanifold $X(G)\times \{x\} \subset X(G) \times \Sigma_g$ for every 6-manifold obtained as a symplectic sum in Theorem \ref{Theorem 19}. By \cite{[Go]}, every manifold of Theorem \ref{Theorem 19} contains a symplectic surface of genus two. The manifold obtained by blowing up a point that is disjoint from the surface of genus two satisfies the hypothesis of Lemma \ref{Lemma M}.

We now mind the possible values of $c_1c_2$. The claim is that all possible values are obtained by the manifolds constructed in Theorem \ref{Theorem 19}. The case $c_1c_2 = 0$ is clear. The claim is equivalent to finding  $e_i, \sigma_i, e_i', \sigma_i'$ $i = 1, 2$ that produce integers $k_+$ and $k_-$ that satisfy
\begin{center}
$6\cdot ((e_1 + 4(g + r)) + \sigma_1 + e_2 + \sigma_2) = 24k_+$ and 
$-6\cdot ((e_1' + 4(g + r)) + \sigma_1' + e_2' + \sigma_2') = 24k_-$.\\
\end{center}

The numbers $e_i, \sigma_i, e_i', \sigma_i'$ $i = 1, 2$ are as in the statement of Theorem \ref{Theorem 19}, and they correspond to the Euler characteristic and signatures of minimal 4-manifolds of Proposition \ref{Proposition 11} used to construct $X(G)$. Such manifolds can be done so that their characteristic numbers satisfy the equalities $(e_1 + \sigma_1 + e_2 + \sigma_2) =  4k_+ + 4(g + r)$ and $(e_1' + \sigma_1' + e_2' + \sigma_2') = 4k_- + 4(g + r)$. Thus, all possible values of $c_1c_2$ are realized by manifolds from Theorem \ref{Theorem 19} and Theorem \ref{Theorem 22}. The corollary now follows from Lemma \ref{Lemma M}.
\end{proof}

\section{Acknowledgements}

R. T. thanks Ronald J. Stern for suggesting useful references that motivated this paper, and Tian-Jun Li for sending him a copy of his preprint \cite{[HL]}. The hospitality of the math department at the University of Toronto during the production of part of this manuscript is much appreciated. R. T. gratefully acknowledges support from the Simons Foundation. J. Y. thanks the math department at the University of Hawaii at Manoa for its hospitality under which part of this paper was written. The authors wish to thank Robert E. Gompf for pointing out the preprint \cite{[A6]} to us, and John Etnyre and Joel Fine for useful observations that helped us improve the paper.

\end{document}